\documentclass[12pt,reqno]{amsart}

\usepackage{amsmath,amsthm,amssymb,amsfonts,mathrsfs,color}
\usepackage{latexsym,esint}
\usepackage{mathrsfs}
\usepackage[colorinlistoftodos,prependcaption,textsize=tiny]{todonotes}

\usepackage[colorlinks=true,urlcolor=blue, citecolor=red,linkcolor=blue]{hyperref}
\usepackage{xcolor}
\usepackage{amsthm} 
\usepackage{latexsym,amsmath,amssymb}
\usepackage{accents}
\usepackage[colorinlistoftodos,prependcaption,textsize=tiny]{todonotes}
\usepackage{a4wide}
\usepackage{mathtools} 
\usepackage{xparse} 
\usepackage{stmaryrd}

\newtheorem{theorem}{Theorem}[section]
\newtheorem{proposition}[theorem]{Proposition}

\newtheorem{lemma}[theorem]{Lemma}

\def\A{{\tilde{A}}}

\def\C{{\mathcal C}}
\def\D{{\mathcal D}}

\def \O{{\Omega}}

\def\I{{\mathcal I}}

\def\M{{\mathcal M}}

\def\e {{\varepsilon}}
\def\le {|\ln {\varepsilon}|}

\newcommand{\RE}{\text{Re}}

\renewcommand{\vec}[1]{\text{\boldmath $#1$}}

\def\dist{{\rm dist\,}}

\def\curl{{\rm curl\,}}

\def\supp{{\rm supp\,}}

\def \p {\partial}
\def \J {{\mathcal J}}

\newcommand{\R}{\mathbb{R}}

\DeclareMathOperator \tr{\rm {tr}}
\DeclareMathOperator \dive{\rm {div}}

\numberwithin{equation}{section}

\title[]{Ginzburg-Landau minimizers with high topological degrees in an annulus}

\date{\today}

\author{Amandine Aftalion}
\address[A.Aftalion]{Universit\'e Paris-Saclay, CNRS,  Laboratoire de math\'ematiques d'Orsay, 91405, Orsay, France}
\email{ amandine.aftalion@math.cnrs.fr}
\author{Rémy Rodiac}
\address[R. Rodiac]{Institute of Mathematics, University of Warsaw,
Banacha 2, 02-097 \\
Warszawa, Poland \(\&\) Laboratoire J.A. Dieudonn\'e, Universit\'e C\^ote d'Azur, CNRS UNMR 7351,06108, Nice, France.}
\email{ rrodiac@mimuw.edu.pl, remy.rodiac@univcotedazur.fr}

\begin{document}

\begin{abstract}
Motivated by recent experiments on fermionic rings, we study the asymptotic behaviour of minimizers of the Ginzburg-Landau (GL) energy in an annulus with a Dirichlet data which depends on the GL parameter on the outer boundary.  We show that there is a critical degree of order $|\ln \e|$  under which the ground state displays a giant vortex and above which minimizers exhibit a combination of a giant vortex and vortices which tend to the outer boundary as the GL parameter tends to zero. Our analysis relies on the construction of suitable upper and lower bounds, on the extension to a slightly bigger annulus and on the minimization of the mean-field energy appearing in the lower bound. In order to be able to derive the minimum of this energy we use the symmetry of the domain and criticality with respect to inner variations.
\end{abstract}


\maketitle

\maketitle

\section{Introduction}

\subsection{Motivation}

 Persistent currents in annular geometries have attracted a lot of attention both for fundamental and experimental reasons. Moreover, they are promising for quantum technologies. Indeed, quantum gases in annuli provide
 unique opportunities to explore quantum many body physics \cite{persistent2024}. Long lived supercurrents have been observed and studied in superfluids (starting with Helium \cite{fetter67}), weakly interacting Bose-Einstein condensates \cite{perrin,zoran} and very recently in fermionic superfluid rings \cite{cai2022,Roati}. These fermionic rings are in a regime of molecular condensate described by the Gross-Pitaevskii equation.

 The geometry of the annulus can also be the source of mathematical questions and, for example, the issue of understanding the behaviour of minimizers of the Ginzburg-Landau (GL) energy in an annulus has been stated as an open problem in the seminal book by  Bethuel-Brezis-H\'elein  ``Ginzburg-Landau vortices", cf.\ \cite[Problem 3, Chap. XI]{BBH}. More precisely, if a boundary data with degree $d_{\text{out}}$ on the outer boundary and with degree \(d_{\text{inn}}\) on the inner boundary is prescribed, the
 issue is to describe minimizers of the GL energy, the location of their zeroes around which the winding-number does not vanish, and the renormalized energy. Some mathematical works allow to give partial answers as we will describe below. 
  
  In this article, in addition to the difficulty of dealing with an annular geometry we consider a high exterior degree, as $\e$, a small parameter, tends to zero. The situation is the following: the Dirichlet data is prescribed only on the outer boundary and its degree grows as \(\e\) tends to zero. Our motivation comes from very recent experiments on fermionic rings, namely the production of persistent currents in annular geometries \cite{cai2022,Roati}. The ring is realized thanks to a tight confinement. Experimentally, a phase of $2\pi d$  is imprinted, which can be  modelled mathematically by a prescribed boundary Dirichlet data $e^{id\theta}$. The specificity of the experiment is that for intermediate values of $d$, vortices are nucleated in the bulk, while for lower values of $d$, only the giant vortex is present. An experimental question raised by \cite{cai2022,Roati} is to understand the critical degree where vortices start to appear in the bulk. In the experiments, $\e$ is not so small that $|\ln \e|$ is large, so that  nucleation occurs for a degree 8. Mathematically, it is easier to deal with the case where $\e$ is small. It is quite surprising that so recent experiments can be described by a simple mathematical formulation where the mechanism of creation of vortices is due to an interaction between the geometry of the domain and the high degree of the Dirichlet boundary data on the exterior boundary. We prove in this paper that there is a critical degree of order $|\ln \e|$  under which the minimizer has a giant vortex and above which vortices also appear in the bulk and go to the outer boundary as $\e$ tends to zero.

 \subsection{Main results}

Motivated by recent experiments in physics \cite{Roati}, we consider a circular annulus \(A=B_{R_2}(0) \setminus \overline{B}_{R_1}(0)\) with \( 0<R_1<R_2\). For small \(\e\), we are interested in the behaviour of minimizers of
\begin{equation}\label{eq:GL_energy}
E_\e(u)=\frac12 \int_A |\nabla u|^2+\frac{1}{4\e^2}\int_A (1-|u|^2)^2
\end{equation}
in a space of maps with degree \(d_\e\) on the outer boundary. A first natural choice of minimization space could be \(\{v\in H^1(A,\mathbb{C}); |v|=1 \text{ on }  \p A \text{ and } \deg (v, \p B_{R_2})=d_\e\}\). However existence and non-existence of minimizers of the GL energy in such a space were studied e.g.\ in \cite{Berlyand_Golovaty_Rybalko_2006,berlmiro,DosSantos_Rodiac_2016,Farina_Mironescu_2013,Golovaty_Berlyand_2002,
Mironescu_2013} and it was shown that existence is not always guaranteed. It depends on the ``size'' of annular domains, or more precisely their capacity \cite{berlmiro}, as do the uniqueness and symmetry of minimizers \cite{Farina_Mironescu_2013,Golovaty_Berlyand_2002}. This is why we work with the minimization space:
\begin{equation*}
\mathcal{I}_\e=\{v\in H^1(A,\mathbb{C}); v=e^{i d_\e\theta} \text{ on } \p B_{R_2}, \ |v|=1 \text{ on } \p B_{R_1}\}
\end{equation*}
where the boundary data  has a topological degree \(d_\e\in \mathbb{N}^*\).
 The direct method in the calculus of variations can be used to prove the existence of minimizers of \(E_\e\) in \(\I_\e\). Furthermore, we can check that such a minimizer satisfies the following Euler-Lagrange equations:
\begin{equation}\label{eq:Euler_Lagrange}
\left\{
\begin{array}{rcll}
-\Delta u_\e &=& \frac{u_\e}{\e^2}(1-|u_\e|^2) & \text{ in } A \\
u_\e &=&e^{i d_\e\theta}  & \text{ on }  \p B_{R_2} \\
|u_\e|&=&1 \text{ and } u_\e\wedge \p_\nu u_\e =0 & \text{ on } \p B_{R_1},
\end{array}
\right.
\end{equation}
 where \(\nu\) denotes the unit outward normal to \( \p A\) and \(a\wedge b=\frac12 (\bar{a}b-a\bar{b})\) for \(a,b\) in \(\mathbb{C}\).

 We aim at describing the location of vortices of minimizers of \(E_\e\) in \(\mathcal{I}_\e\) in \(A\) in the limit \(\e\to 0\) with the degree \(d_\e\) going to \(+\infty\). In order to describe this situation we use the current vector and the Jacobian vorticity measure of a family of minimizers \((u_\e)_{\e>0}\). They are defined by
\begin{equation}\label{def:current_vorticity}
j_\e:=u_\e\wedge \nabla u_\e=(u_\e\wedge \p_x u_\e,u_\e\wedge \p_y u_\e) ,\quad \mu_\e:= 2\p_x u_\e \wedge \p_y u_\e=\curl j_\e.
\end{equation}
 Our first result shows that there exists a critical degree under which, for \(\e\) small enough, minimizers of \(E_\e\) in \(\I_\e\) are vortex-less and resemble a giant vortex of degree \(d_\e\).
 
 \begin{theorem}\label{th:uniqueness_d_small}
 Let \( (u_\e)_\e\) be a family of minimizers of \(E_\e\) in \({\I}_\e\). Let 
 \begin{equation}\label{eq:critical_alpha}
 \alpha_c:=\frac{1}{2\ln \left(R_2/R_1 \right)}.
 \end{equation}
  We recall that  \(\vec{e_\theta}=(-\sin \theta,\cos \theta)\).
  \begin{itemize}
  \item[1)]Assume that \(d_\e< \alpha_c |\ln \e|\). Then 
 \begin{equation}
 \frac{u_\e}{e^{id_\e\theta}} \rightarrow 1, \text{ in } \C^0_{\text{loc}}(\overline{B}_{R_2}\setminus \overline{B}_{R_1}).
 \end{equation}
 
 \item[2)] Assume that \(d_\e/|\ln \e| \to \alpha \) for some \(\alpha \in \R^+\) with \(\alpha\leq \alpha_c\). Then \(\lim_{\e\to 0} \frac{E_\e(u_\e)}{|\ln \e|^2}=\pi \alpha^2 \ln \frac{R_2}{R_1}\),
\begin{equation*}
\frac{j_\e}{|\ln \e|}\rightharpoonup \frac{\alpha \vec{e_\theta}}{r} \text{ in } L^2(A)  \text{ and }  \frac{\mu_\e}{|\ln \e|} \rightharpoonup 0 \text{ in } (\C^{0,\gamma}_c(\overline{B}_{R_2}\setminus \overline{B}_{R_1}))^* \text{ for all } \gamma \in (0,1).
\end{equation*}

  \end{itemize}
 
  \end{theorem}

 We point out that $\alpha_c$ is large for a small capacity. 
 Our second theorem states that, when \(d_\e=\alpha|\ln \e|\) with \(\alpha>\alpha_c\) then vortices appear in the domain and they are close to the exterior boundary. 

\begin{theorem}\label{th:main1}
Let \( (u_\e)_\e\) be a family of minimizers of \(E_\e\) in \({\I}_\e\). Assume that \(d_\e/|\ln \e| \to \alpha \) for some \(\alpha \in \R^+\). Recall that \(\alpha_c\) is defined in \eqref{eq:critical_alpha}.  If \(\alpha> \alpha_c\)  then :


\begin{align*}
\lim_{\e\to 0} \frac{E_\e(u_\e)}{|\ln \e|^2}& =\pi (\alpha-\frac{1}{4\ln \frac{R_2}{R_1}}), \\
\frac{j_\e}{|\ln \e|} & \rightharpoonup \frac{ \vec{e_\theta}}{2\ln (R_2/R_1) r} \text{ in } L^2(A) \text{ and } \\
\frac{\mu_\e}{|\ln \e|} & \rightharpoonup \left(\alpha-\frac{1}{2\ln \frac{R_2}{R_1}}\right)\mathcal{H}^1_{\lfloor \p B_{R_2}}\text{ in }(\C^{0,\gamma}_c(\overline{B}_{R_2}\setminus \overline{B}_{R_1}))^* \text{ for all } \gamma \in (0,1).
\end{align*}
\end{theorem}

  This result states that, when \(\alpha >\alpha_c\) minimizers look like a combination of a giant vortex of degree  of order \( (1/2\ln (R_2/R_1))|\ln \e|\) plus a sum of approximately \( d_\e -1/2\ln (R_2/R_1)|\ln \e|\) vortices located near \(\p B_{R_2}\). However, although these vortices converge to the boundary \(\p B_{R_2}\) they are at a distance much bigger than \(\e\) of it. Hence they do not cost twice as much as inner vortices as this is the case for vortices at a distance less than \(\e\) from \(\p B_{R_2}\).

We observe that in the case of rotating superfluids, the situation is similar but different: at low velocity, there is a giant vortex and as the velocity is increased, vortices are nucleated in the bulk \cite{fetter67,perrin} but they do not approach the outer boundary. Indeed the trapping potential allows an equilibrium position for the circle of vortices inside the condensate \cite{aab}. This takes place at a critical rotational velocity of order $\le+ \omega  \ln \le$.

There are two difficulties emerging in our problem which are the annular geometry and the large number of vortices. The proof of Theorem \ref{th:uniqueness_d_small} makes use of a splitting of the energy obtained by writing minimizers as \(u_\e=\eta_\e e^{id_\e\theta}v\) where \(\eta_\e\) is the minimizer of the energy 
 among real-valued functions in \(H^1( (R_1,R_2))\) satisfying \(\eta(R_1)=\eta(R_2)=1\). This decomposition, originally due to \cite{Lassoued_Mironescu_1999}, allows to write \(E_\e(u_\e)=E_\e(\eta_\e e^{id_\e\theta})+G_\e(v)\) where \(G_\e(v)\) is a weighted GL energy with an interaction term whose sign is not controlled. In order to prove that, if \(u_\e\) minimizes \(E_\e\), then \(G_\e(v)=0\) and \(v=1\) we use a weighted Jacobian estimate due to Jerrard \cite{Jerrard_2007} in the same way it was employed in \cite{Aftalion_Jerrard_RoyoLetelier_2011} to prove uniqueness of the real-valued minimizer in a Bose-Einstein condensate under low-rotation. 

 The proof of Theorem \ref{th:main1} consists in constructing good test functions to obtain an upper bound on the value of the minimum of the energy and to derive a corresponding lower bound. The upper bound is built with a giant vortex of degree of order $ \beta|\ln \e| $ and about $ (\alpha -\beta) |\ln \e|$ vortices located on a circle close to the outer boundary and optimizing on $\beta$. For the lower bound, we use the vortex-balls construction of Sandier \cite{Sandier_1998} and Jerrard \cite{Jerrard_1999}, and the Jacobian estimates of Jerrard-Soner \cite{Jerrard_Soner_2002}. A key tool to understand the concentration phenomena close to the boundary is to extend the wave function in a slightly larger annulus before applying the estimates.

Another approach to tackle this problem could be to use the conformal transform between an annulus and a finite cylinder. This has been analyzed by Fetter et al \cite{fetter67,Guenther_Massignan_Fetter_2017} using image vortices: vortices located on a circle inside the annulus are mapped to vortices on a line in a cylinder. For the
finite cylinder, the optimal wave function is constructed for each vortex, by taking into
account a series of image vortices. This changes the decay of the wave function at
infinity from algebraic to exponential, see also \cite{aftsan}. 
 We have not used this conformal transform here but it is likely that this approach and image vortices could be used to better understand the interaction with the boundary and the energy expansion. At least extending the wave function to compute a proper
interaction term seems important.

%
%

In the case of large degree, that is \(d_\e \gg |\ln \e|\), at leading order, vortices accumulate close to the outer boundary:
 \begin{proposition}\label{th:main_0_case_d_big}
 Let \( (u_\e)_\e\) be a family of minimizers of \(E_\e\) in \(\I_\e\). Assume that there exists \(\delta>0\) such that \( |\ln \e| \ll |d_\e| \ll \e^{\delta-1}    \). Then,
 \(j_\e/d_\e \rightarrow 0\) in \(L^2(A)\). Furthermore if we let \(a>0\) and if we define
\begin{equation*}
\tilde{\mu}_\e:=\begin{cases}
\mu_\e & \text{ in } A, \\
0 & \text{ in } B_{R_2+a}\setminus B_{R_2},
\end{cases}
\end{equation*} 
  then \(\tilde{\mu}_\e/d_\e\rightharpoonup \mathcal{H}^1_{\lfloor \p B_{R_2}}\) in \(\mathcal{D}'(A_a)\).
 \end{proposition} 
 We observe that the convergence obtained in this proposition is much weaker than in the previous theorems. But it roughly indicates that the vorticity is concentrated near the boundary. Although it does not appear at first order it is possible that there is still a giant vortex of lower order on the inside boundary.

\subsection{Main ideas}

In their pioneering work \cite{BBH}, Bethuel-Brezis-H\'elein studied the emergence of vortices for the Ginzburg-Landau (GL) energy
 \eqref{eq:GL_energy} in a smooth star-shaped domain \(\O\subset \R^2\) with a fixed boundary data \(u=g\in \C^1(\p \O,\mathbb{S}^1)\) having a non zero topological degree \(\deg (g,\p \O)=d\neq 0\). They obtained in \cite{BBH} that for small values of the parameter \(\e\), minimizers of the energy have exactly \(d\) vortices of degree one whose limiting locations minimize a renormalized energy. We also refer to \cite{Struwe_1994}, \cite{delPino_Felmer_1998} for an extension of these results to simply-connected domains. We point out that all the limiting vortices are inside the domain \(\O\). 
 
 Generalizing this result to the framework of annuli and a large degree is the aim of this paper. 
 The main tool to handle this problem is the vortex-balls construction independently devised by Jerrard \cite{Jerrard_1999} and Sandier \cite{Sandier_1998}. Thanks to this vortex-balls construction, the behaviour of minimizers with large GL energy can be understood. For example, this is the case of the GL energy  with exterior magnetic field \(h_{\text{ex}}\) of order \(|\ln \e|\), many vortices may appear inside the domain and their location is related to the obstacle problem, see \cite{Sandier_Serfaty_2007} and references therein. Actually the vortex-balls construction allows to treat also the case of non simply connected domain as observed e.g.\ in \cite{Jerrard_Soner_2002}, \cite{Alama_Bronsard_2006} and also to treat vacuum manifolds \(\mathcal{N}\) which are different from \(\mathbb{S}^1\), see e.g.\ \cite{Monteil_Rodiac_VanSchaftingen_2021}, \cite{Monteil_Rodiac_VanSchaftingen_2022}. The renormalized energy for non simply connected domains was derived in \cite{delPino_Kowalczyk_Musso_2006} and \cite{Rodiac_Ubillus_2022}. In particular the results of \cite{Monteil_Rodiac_VanSchaftingen_2021,Monteil_Rodiac_VanSchaftingen_2022,Rodiac_Ubillus_2022}, allow to obtain results analogous to the one of Bethuel-Brezis-Hélein \cite{BBH} for multiply connected domains. But the extra difficulty in our paper is the large degree.

In \cite{Sandier_Soret_2000}, Sandier-Soret considered the limit of minimizers of the Bethuel-Brezis-Hélein renormalized energy, and proved that the singularities of harmonic maps 
minimizing the renomalized energy go to the boundary when their number becomes large, that is when the degree  $d$ tends to \(+\infty\). This corresponds to studying firstly the  limit \(\e\to 0\)  to get the renormalized energy and then to let the number of points or the degree $d$ tend to \( +\infty\). They found  that the empirical probability measures \(\frac{1}{d}\sum_{i=1}^d \delta_{x_i}\) always converge to a measure supported on the boundary of the domain. Since in our case, the degree is large and depends on $\e$, this is a double limit and we cannot apply their proof as such. Their limiting measure only contains an interaction term while ours encompasses a mass term as well, since in our case the number of vortices is not fixed a priori and is an unknown of the problem.  

 The GL energy  in a regime allowing for a diverging number of vortices was studied by Jerrard-Soner \cite{Jerrard_Soner_2002}. More precisely they studied the \(\Gamma\)-limit of \(E_\e/|\ln \e|^2\). We observe that this is exactly the energy scaling of Theorem \ref{th:main1}. However their results cannot be applied directly to our setting because they did not consider any boundary data. There also exist  other results about solutions to the GL equations with a diverging number of vortices. For example, the asymptotic behaviour of  critical points of \(E_\e\) was studied by Sandier-Serfaty in \cite{Sandier_Serfaty_2003} and \cite[Chapter 7]{Sandier_Serfaty_2007}. They obtained that if the limit of the vorticity measures \(\mu_\e/\sqrt{E_\e(u_\e)}\) belongs to \(L^p(\Omega)\), where \(\Omega\) is a smooth simply connected domain, for some \(p>1\) then necessarily \(\mu=0\). This means that if many vortices appear, they have to concentrate on singular sets (with respect to the Lebesgue measure). This was later extended in \cite{Rodiac_2016} where it is shown that actually, limiting vorticity measures can only concentrate on a union of (possibly intersecting) smooth curves. Thus we expect a priori that the limit of \(\mu_\e/\sqrt{E_\e(u_\e)}\) in the sense of measures will be supported on curves possibly included in the boundary. We  obtain a more precise result by the construction of an upper bound and a lower bound.  More specifically, we obtain that when vortices appear, the curve on which they concentrate is the outer boundary. In fact, to understand this phenomenon of concentration, we need to extend the boundary data slightly away from the outer boundary. We also point out that it is crucial for us to work with a circular annulus since it implies that the curves where vortices concentrate can only be circles and this reduces the problem to a 1D problem.
\medskip

The paper is organized as follows: in Section \ref{sec:upper_Bound} we provide upper bounds for the value of the minimum of the energy \(E_\e\) in \(\I_\e\). The main upper bound is obtained by constructing test functions which are made of a giant vortex of degree \(\gamma_\e\leq d_\e\) and \(d_\e-\gamma_\e\) vortices near the boundary. We  optimize this upper bound with respect to the number of vortices and find the critical value such that for \(d_\e \geq \left(1/2\ln (R_2/R_1)\right) |\ln \e|\) then it is preferable to have vortices near \(\p B_{R_2}\). In Section \ref{sec:tools} we recall the vortex-balls construction of Jerrard \cite{Jerrard_1999} and Sandier \cite{Sandier_1998} and the Jacobian estimates of Jerrard-Soner \cite{Jerrard_Soner_2002_Calvar}. These tools are tailored to describe vortices in the interior of the domain. In order to capture the possible convergence of vortices to the boundary we need to extend minimizers to a slightly bigger annulus. The most difficult case is the case where \(d_\e\) is of order \(|\ln \e|\) and is treated in Section \ref{sec:d_order_lneps} where we prove Theorem \ref{th:main1} and obtain the second item of Theorem \ref{th:uniqueness_d_small}. Section \ref{sec:d_small} is devoted to the proof of the first item of Theorem \ref{th:uniqueness_d_small}, namely we show the behaviour of  the minimizer when \(d_\e< \alpha_c|\ln \e|\) and \(\e\) is small enough.  On the contrary many vortices appear near the boundary \(\p B_{R_2}\) when \(d_\e \gg |\ln \e|\) and this is the object of  Section \ref{sec:d_big}.

\subsection{Notations} Throughout this paper we will use the following notation:
\begin{itemize}
\item[-] The vectors of the moving basis in polar coordinates \( (r,\theta)\in \R^+\times [0,2\pi)\) are denoted by \( \vec{e_r}=(\cos \theta,\sin \theta)\) and \(\vec{e_\theta}=(-\sin \theta, \cos \theta)\).
\item[-] If \(\Omega\) is a subset of \(\R^d\) with \(d\geq 1\) then we denote by \(\M(\Omega)\)  the set of bounded Radon measures on \(\Omega\). The set of non-negative Radon measures on \(\Omega\) is denoted by \(\M(\Omega)^+\). We will frequently use the weak convergence in sense of measures: if \( (\mu_n)_n\) is a sequence in \(\M(\O)\) and \(\mu\in \M(\Omega)\) we say that \(\mu_n\) converges weakly in the sense of measure towards \(\mu\) and we denote \(\mu_n \rightharpoonup \mu\) in \(\M(\Omega)\) if 
\[ \int_{\Omega} \varphi d \mu_n \xrightarrow [n \to +\infty]{} \int_{\Omega} \varphi d \mu, \text{ for any } \varphi \in \C_c(\Omega).\]
For \(\mu\in \M(\Omega)\) we denote by \(|\mu|\) the total variation measure of \(\mu\) and by \(\|\mu\|:=|\mu|(\Omega)\) its total variation.
\item[-] For two complex numbers \(a,b\in \mathbb{C}\), we define \(a\wedge b=\frac{1}{2i} (\bar{a}b-a\bar{b})\).
\item[-] For \(x\in \R\), we denote by \(\lfloor x \rfloor \) the integer part of \(x\), i.e., the smallest integer such that \(\lfloor x \rfloor\leq x <\lfloor x \rfloor+1\).
\item[-] We write \(a_\e \ll b_\e\) when \(\e \to 0\) for \(a_\e=o(b_\e)\).
\end{itemize}

\subsection{Acknowledgements:} This paper originated with the CNRS 80 prime project Tradisq1d.
We are very grateful to Fr\'ed\'eric Chevy for explaining the physics behind. We warmly thank Etienne Sandier and Bob Jerrard
 for  enlightening discussions concerning this work. The research of R.R is part of the project 
 No. 2021/43/P/ST1/01501
 co-funded by the National Science Centre and the European Union Framework Programme for Research and Innovation Horizon 2020 under the Marie Skłodowska-Curie grant agreement No. 945339. The research of AA is supported by France 2030 PIA funding: ANR-21-EXES-0003.

\section{Upper bounds}\label{sec:upper_Bound}

In this section we construct two upper bounds for \(\min_{v\in \I_\e} E_\e(v)\). Each of these upper bounds will be optimal for some regimes of \(d_\e\). The first one is a giant vortex and the second one with a circle of vortices inside the bulk.

\begin{lemma}\label{lem:1st_upper_bound}
Let \(u_\e\) be a minimizer of \(E_\e\) in \(\I_\e\) then 
$$ E_\e(u_\e)\leq  \pi d_\e^2\ln \frac{R_2}{R_1}\left(1-\frac {\e^2d_\e^2}{8\ln \frac{R_2}{R_1}}\left (\frac 1 {R_1^4}- \frac 1 {R_2^4} \right )+O(\e^4 d_\e^4)\right) \leq \pi d_\e^2\ln \frac{R_2}{R_1}.$$
\end{lemma}

\begin{proof} A first upper bound can be obtained by taking as a test function \(v_\e=e^{id_\e \theta}\) and observing that \(E_\e(v_\e)= \pi d_\e^2\ln \frac{R_2}{R_1}\). A more refined one is to restrict the  minimization of  \(E_\e\) to functions of the type $\eta e^{id_\e \theta}$ where $\eta$ is real valued, $\eta(R_1)=\eta(R_2)=1$. Then the energy is 
\begin{equation}\label{eq:energy_eta}
I_\e (\eta) =\frac12 \int_A (\nabla \eta)^2+\frac12 \int_A \frac{d_\e^2}{r^2} \eta^2+\frac{1}{4\e^2}\int_A (1-\eta^2)^2.
\end{equation}
Note that this can be rewritten as
\begin{equation}\label{eneta}
  \frac12 \int_A (\nabla \eta)^2+\frac{1}{4\e^2}\int_A \left (1-\frac{d_\e^2\e^2}{r^2}-\eta^2\right )^2 +  \int_A \left (-\frac{\e^2d_\e^4}{4r^4}+\frac{d_\e^2}{2r^2} \right ) .
\end{equation} The last term is a constant term which provides the leading order term of the estimate.
 We can use  as a test function $\eta^2=1-d_\e^2 \e^2/r^2$ matched to 1 at $R_1$ and $R_2$ on a size of order $d_\e^2 \e^2$, which leads the estimate, the leading order term coming from the second term. 
\end{proof} 
Thanks to the formulation \eqref{eneta}, proceeding as in \cite{Aftalion_2006}, we can obtain that the minimizer $\eta_\e$ among real functions is uniformly close to 1 with an error of order $\e^2 d_\e^2$. We see that as soon as $d_\e\ll 1/\e$, then minimizing among real functions or taking 1 is equivalent at leading order.

 A more refined upper bound consists in putting a giant vortex of degree $\gamma_\e$ and $N_\e$ vortices on a circle close to the outer boundary:
\begin{proposition}\label{prop:2nd_upper_bound}
Assume that \(d_\e \gg 1\) and let \(u_\e\) be a minimizer of \(E_\e\) in \(\I_\e\). For any \(\gamma_\e\) in \(\mathbb{N}\) with \(\gamma_\e\leq d_\e\) and \(\gamma_\e \gg 1\), if \(N_\e:=d_\e-\gamma_\e\) satisfies that \( 1 \ll N_\e \ll \e^{-1}\), then there exists a universal constant \(C>0\) such that
\begin{equation}\label{eq:upperbound}
E_\e(u_\e)\leq \pi \gamma_\e^2\ln \frac{R_2}{R_1}+\pi N_\e (|\ln \frac 1{\e N_\e}| +C)+\pi \gamma_\e N_\e
\ln \frac{1}{1-\frac {\pi}{2N_\e}}+ O(d_\e \e +d_\e^3 \e^2).
\end{equation}

\end{proposition}

The idea to construct this upper bound is to take test-functions of the form \(v_\e=e^{i\gamma_\e \theta} w_\e\) with \(w_\e\) which is equal to one in most of the annulus except near the boundary \(\p B_{R_2}\). Furthermore \(w_\e\) has \(N_\e=d_\e-\gamma_\e\) vortices of degree 1 equi-distributed on a circle of radius \(R_2(1-\frac{\pi}{2N_\e})>R_1\). The first term in the upper bound is the energy of the giant vortex, the next term is the energy of the vortex balls and the last term is the interaction energy between the giant vortex and the individual vortex balls.

We start by showing a decoupling lemma for functions of the form \(e^{i\gamma_\e \theta} w_\e\).

\begin{lemma}\label{lem:decoupling}
Let \(v_\e=e^{i\gamma_\e \theta} w_\e\) be in \(\I_\e\) with \(\gamma_\e\) in \( \mathbb{Z}\) and \(|\gamma_\e|\leq d_\e\). Then,
\begin{equation}\label{eq:decoupling}
E_\e(v_\e)=\pi \gamma_\e^2\ln \frac{R_2}{R_1}+E_\e(w_\e)+\gamma_\e \int_{A} \frac{w_\e\wedge \p_\theta w_\e}{r^2}+ \int_{A} (|w_\e|^2-1)\frac{\gamma_\e^2}{r^2}.
\end{equation}
\end{lemma}

\begin{proof}
We compute
\begin{equation*}|\nabla v_\e|^2=
 |w_\e|^2|\nabla e^{i\gamma_\e\theta}|^2+|\nabla w_\e|^2+2\gamma_\e \frac{w_\e\wedge \p_\theta w_\e}{r^2}. \nonumber
\end{equation*}
This leads to
\begin{equation*}
E_\e(v_\e)=E_\e(w_\e)+\frac12 \int_{A} |w_\e|^2 |\nabla (e^{i\gamma_\e \theta})|^2+ \gamma_\e \int_{A} \frac{w_\e\wedge \p_\theta w_\e}{r^2}.
\end{equation*}
The conclusion follows by observing that
\[ \frac12 \int_{A} |w_\e|^2 |\nabla (e^{i\gamma_\e \theta})|^2 =\int_A |\nabla (e^{i\gamma_\e \theta})|^2 +\int_{A} (|w_\e|^2-1)|\nabla (e^{i\gamma_\e \theta})|^2.\]
\end{proof}

We are now ready to prove Proposition \ref{prop:2nd_upper_bound}.

\begin{proof}(Proof of Proposition \ref{prop:2nd_upper_bound})
Let \(\gamma_\e\) be in \(\mathbb{N}\) with \(1\ll\gamma_\e\leq |d_\e|\). We start by defining \(N_\e=d_\e-\gamma_\e\) and \( D_k:=B\left(R_2e^{i\left(\frac{\pi}{N_\e}+\frac{2k\pi}{N_\e}\right)},\frac{\pi R_2}{N_\e}\right)\cap A\) for \(k=0,\dots,N_\e-1\). We take as test functions \(v_\e=e^{i\gamma_\e \theta}w_\e\) where
\begin{equation*}
w_\e=1 \text{ in } A \setminus \bigcup_{k=0}^{N_\e-1} D_k.
\end{equation*}
It remains to define \(w_\e\) in each \(D_k\). We prescribe \(w_\e=1\) on \(\p D_k\cap A\) and \(w_\e=e^{iN_\e \theta}\) on \(\p D_k\cap \p A\). Note that each \(D_k\) is obtained from \(D_0\) by a rotation of angle \(2k\pi/N_\e\). Since the GL energy is rotationally invariant we can estimate the energy on \(D_0\) and then sum \(N_\e\) times this energy. The set \(D_0\) is conformally equivalent to \(B(0,\frac{1}{N_\e})\) via a conformal map \(\Phi: B(0,\frac{1}{N_\e})\rightarrow D_0\). We assume that  \(\Phi(0)=R_2(1-\frac{\pi}{2N_\e})e^{\frac{i\pi}{N_\e}}\). We can also assume that \(\Phi\) is  bi-Lipschitz up to the boundary with \(\|\Phi\|_{\text{Lip}}+\| \Phi^{-1}\|_{\text{Lip}}\leq M\) for some constant \(M>0\) independent of \(\e\), cf.\ \cite[Theorem 3.5]{Pommerenke_1992}. Thus if we set \(\tilde{w}_\e=w_\e\circ \Phi\) we have
\begin{align}
E_\e(w_\e,D_0)&=\frac12 \int_{D_0}|\nabla w_\e|^2 +\frac{1}{4\e^2}\int_{D_0} (1-|w_\e|^2) \nonumber \\
&= \frac12 \int_{B(0,\frac{1}{N_\e})}|\nabla \tilde{w}_\e|^2 +\frac{1}{4\e^2} \int_{B(0,\frac{1}{N_\e})} (1-|\tilde{w}_\e|^2)|\det D \Phi| \nonumber \\
& \leq \frac12 \int_{B(0,\frac{1}{N_\e})}|\nabla \tilde{w}_\e|^2 +\frac{M^2}{4\e^2} \int_{B(0,\frac{1}{N_\e})} (1-|\tilde{w}_\e|^2). \label{eq:Meps'}
\end{align}
We set \(\e'=\e/M\), we call \(U_{\e'}\) the radial minimizer\footnote{For the existence and uniqueness of such a minimizer we refer to \cite{Herve_Herve_1994} or \cite[Appendix II]{BBH}} of \( E_{\e'}\) in \(B(0,\e')\) with boundary data \(e^{i\theta}\) and, for \(0<\e'\leq \rho<1/N_\e\) we define \(V_{\rho}\) as the minimizer\footnote{Here also the minimizer exists and is unique, see \cite[Chapter 1]{BBH}} of
\begin{multline*}
\inf \Bigl\{ \int_{B(0,\frac{1}{N_\e})\setminus B(0,\rho)}|\nabla v|^2; v\in H^1\left(B\left(0,\frac{1}{N_\e}\right)\setminus B(0,\rho),\mathbb{S}^1\right), \\
v=e^{i\theta} \text{ on } \p B(0,\rho) \text{ and } v=g_\e\circ \Phi \text{ on } \p B\left(0,\frac{1}{N_\e}\right) \Bigr\}.
\end{multline*}
Here \[g_\e=\begin{cases}
1 & \text{ on } \p D_0 \cap A \\
e^{i N_\e \theta} & \text{ on } \p D_0 \cap \p A
\end{cases}.\]

We observe that if we parametrize the boundary \(\p D_0\) in such a way that the derivative of the parametrization is constant on \(\p A \cap \overline{D}_0\) with \(\tau=0\) corresponding to the point \(R_2\) and \(\tau=\pi\) corresponding to the point \(R_2e^{\frac{2i\pi}{N_\e}}\) 
and \(\tau \in (\pi,2\pi)\) on \(\p D_0\cap A\), then 
\[g_\e=\begin{cases}
1 & \text{ for } \tau \in (\pi,2\pi) \\
e^{i\tau } & \text{ for } \tau \in (0,\pi).
\end{cases}.\]
Hence there exists \(C>0\) independent of \(\e\) such that \(|\p_\tau g_\e|\leq C\). We can use this and the Lipschitz character of \(\Phi\)  to obtain that 
\begin{equation}\label{eq:bound_derivative_g_phi}
|\p_\theta \left(g_\e\circ \Phi\left(\frac{e^{i\theta}}{N_\e} \right)\right)|\leq C \text{ for } \theta \in (0,2\pi).
\end{equation}
We then set
\begin{equation*}
\tilde{w}_\e(y)=\begin{cases}
U_{\e'} &\text{ for } y\in B(0,\e') \\
V_{\e'} & \text{ for } y \in B(0,\frac{1}{N_\e})\setminus \overline{B}(0,\e').
\end{cases}
\end{equation*}
By scaling we can see that \( E_{\e'}(U_{\e'},B(0,\e')) =E_1(U_1,B(0,1))\leq C\) for some universal constant \(C>0\). Furthermore, proceeding for example as in \cite[Proposition 2.9]{Monteil_Rodiac_VanSchaftingen_2022}, we obtain that for \(\e'<\rho<(1/N_\e)\) we have
\begin{equation}\label{eq:compa_monotonie}
\frac12 \int_{B(0,\frac{1}{N_\e})\setminus B(0,\e')}|\nabla V_{\e'}|^2-\pi \ln \frac{1}{\e'}\leq \frac12 \int_{B(0,\frac{1}{N_\e})\setminus B(0,\rho)}|\nabla V_\rho|^2-\pi \ln \frac{1}{\rho}.
\end{equation}
Indeed it suffices to define \( \tilde{V}_{\e'}=\begin{cases} V_\rho & \text{ on } B(0,\frac{1}{N_\e})\setminus B(0,\rho), \\
\frac{x}{|x|} & \text{ on } B(0,\rho)\setminus B(0,\e')\end{cases}\) and to observe that the minimality of \(V_{\e'}\) gives that 
\begin{align*}
\frac12 \int_{B(0,\frac{1}{N_\e})\setminus B(0,\e')}|\nabla V_{\e'}|^2 &\leq \frac12 \int_{B(0,\frac{1}{N_\e})\setminus B(0,\e')} |\nabla \tilde{V}_{\e'}|^2 \\
&\leq \int_{B(0,\frac{1}{N_\e})\setminus B(0,\rho)}|\nabla V_\rho|^2+\pi \ln \frac{\rho}{\e'}.
\end{align*}
Now we take \(\rho=\frac{1}{2N_\e}\) in \eqref{eq:compa_monotonie} and we show that there exists \(C>0\) independent of \(\e\) such that 
\begin{equation}\label{eq:comp_V_1surn}
\int_{B(0,\frac{1}{N_\e})\setminus B(0,\frac{1}{2N_\e})}|\nabla V_{\frac{1}{2N_\e}}|^2 \leq C.
\end{equation}
In order to prove \eqref{eq:comp_V_1surn} we use the minimality of \(V_{\frac{1}{2N_\e}}\) and we compare its energy with a suitable competitor. Since \((g_\e\circ \Phi)\) is Lipschitz and is of degree 1 on \(\p B(0,\frac{1}{N_\e})\) we can find \(\varphi_\e\in W^{1,\infty}(\p B(0,\frac{1}{N_\e}))\) such that \( g_\e\circ \Phi=\frac{x}{|x|}e^{i\varphi_\e}\) on \( \p B(0,\frac{1}{N_\e})\). By using \eqref{eq:bound_derivative_g_phi} we can see that, with some abuse of notations, \(|\p_\theta \varphi_\e|\leq C\) on \( \p B(0,\frac{1}{N_\e})\). Now we take \(\eta_\e\in \C^\infty(\R)\) such that \(\eta_\e \equiv 1\) when \(r\geq (1/N_\e)\) and \(\eta_\e\equiv 0\) when \(r \leq (1/2N_\e)\) with \(|\nabla \eta_\e|\leq C N_\e\). We can compute that 
  \begin{align*}
  \int_{B(0,\frac{1}{N_\e})\setminus B(0,\frac{1}{2N_\e})}|\nabla V_{\frac{1}{2N_\e}}|^2 & \leq \int_{B(0,\frac{1}{N_\e})\setminus B(0,\frac{1}{2N_\e})} |\nabla (e^{i\theta+\eta_\e(r)\varphi(\theta)})|^2 \\
  & \leq C N_\e^2 \left( \frac{1}{N_\e^2}-\frac{1}{4N_\e^2}\right) \leq C.
  \end{align*}
We thus obtain \eqref{eq:comp_V_1surn} and then we can conclude with \eqref{eq:compa_monotonie} that there exists \(C>0\) independent of \(\e>0\) such that 
\begin{equation*}
E_{\e'}(V_{\e'}, B(0,\frac{1}{N_\e})\setminus \overline{B}(0,\e') ) \leq \pi \ln \frac{1}{\e' N_\e}+C.
\end{equation*}
Then by using \eqref{eq:decoupling}, \eqref{eq:Meps'} and by summing over \(k=0,\dots,N_\e-1\) we obtain that
\begin{multline}\label{eq:upper_int}
E_\e(v_\e)\leq \pi \gamma_\e^2\ln \frac{R_2}{R_1}+N_\e (\pi \ln \frac{1}{\e N_\e}+C)+\gamma_\e \int_{A} \frac{w_\e\wedge \p_\theta w_\e}{r^2}\\
+ \int_{A} (|w_\e|^2-1)|\nabla (e^{i\gamma_\e \theta})|^2.
\end{multline}
We now focus on the two last terms of the right-hand side. In order to deal with \(\int_{A} \frac{w_\e\wedge \p_\theta w_\e}{r^2}\) we observe that, since \(w_\e=1\) outside \(\cup_{k=0}^{N_\e-1}D_k\) we have

\begin{align*}
\int_{A} \frac{w_\e\wedge \p_\theta w_\e}{r^2} &=\sum_{k=0}^{N_\e-1} \int_{D_k} w_\e\wedge \nabla w_\e \cdot \nabla^\perp \ln\frac{|x|}{R_2} \\
&= \sum_{k=0}^{N_\e-1} \Bigl( -\int_{D_k} 2\p_x w_\e \wedge \p_y w_\e \ln\frac{|x|}{R_2} + \int_{\p D_k} w_\e\wedge \p_\tau w_\e \ln \frac{|x|}{R_2}\Bigr).
\end{align*}
Now since \(\ln \frac{|x|}{R_2} =0\) on \( \p D_k \cap \p B_{R_2}\) and \( w_\e\wedge \p_\tau w_\e=0\) on \(\p D_k \cap A\) we arrive at
\[\int_{A} \frac{w_\e\wedge \p_\theta w_\e}{r^2} = \sum_{k=0}^{N_\e-1}\int_{B_k} 2\p_x w_\e \wedge \p_y w_\e \ln\frac{R_2}{|x|}. \]
But actually, \(w_\e\) is of modulus one in almost all \(D_k\) except for a small set which is the image by a Lipschitz map of a ball of radius \(\e/\sqrt{M}\) for some \(M>0\). We call \(\tilde{D}_k\) these small sets and we observe that they have an area bounded by \(C\e^2\) and a diameter bounded by \(C\e\). Thus, by the mean value theorem, in each \(\tilde{D}_k\) we have \( \ln \frac{|x|}{R_2} =\ln \frac{R_2(1-\frac{\pi}{2N_\e})}{R_2} +O(\e)\). By using that \(|\nabla w_\e|\leq C/\e\) in \(\tilde{D}_k\) we find

\begin{align*}
\int_{D_k} 2\p_x w_\e \wedge \p_y w_\e \ln\frac{R_2}{|x|} &= \int_{\tilde{D}_k } 2\p_x w_\e \wedge \p_y w_\e \ln\frac{R_2}{|x|} \\
&= \ln \frac{R_2(1-\frac{\pi}{2N_\e})}{R_2}\int_{\tilde{D}_k} 2\p_x w_\e\wedge \p_yw_\e +O(\e).
\end{align*}
By integrating by parts again and by observing that \(w_\e\) has degree one on each \(\tilde{D}_k\) we find
\[ \int_{D_k} 2\p_x w_\e \wedge \p_y w_\e \ln\frac{R_2}{|x|} = \pi \ln \frac{R_2}{R_2(1-\frac{\pi}{2N_\e})}+O(\e).\]
Thus we can sum and obtain
\begin{equation}\label{eq:interaction_giant_vortex}
\int_{A} \frac{w_\e\wedge \p_\theta w_\e}{r^2} =\pi N_\e \ln \frac{R_2}{R_2(1-\frac{\pi}{2N_\e})}+O(N_\e \e).
\end{equation}

We now consider the term \(\int_{A} (|w_\e|^2-1)|\nabla (e^{i\gamma_\e \theta})|^2\). By using that \(w_\e\) is of modulus one except in the set \(\tilde{D}_k\) which have an area bounded by \(C\e^2\) we readily obtain
\begin{align}
\int_{A} (|w_\e|^2-1)|\nabla (e^{i\gamma_\e \theta})|^2&=\sum_{k=0}^{N_\e-1} \int_{\tilde{B}_k }(|w_\e|^2-1)|\nabla (e^{i\gamma_\e \theta})|^2  \leq N_\e \gamma_\e^2 \e^2 \leq d_\e^3\e^2. \label{eq:remainder_coupling}
\end{align} Therefore, this term has no contribution.
Hence by \eqref{eq:upper_int}, \eqref{eq:interaction_giant_vortex} and \eqref{eq:remainder_coupling} we find that \eqref{eq:upperbound} holds.

\end{proof}

The next step is to optimize the upper bound \eqref{eq:upperbound} in terms of \(N_\e\). 


\begin{proposition}\label{cor:upper_bound_final}
Let \(u_\e\) be a minimizer of \(E_\e\) in \(\I_\e\).
\begin{itemize}
\item[1)]  If \( d_\e \leq \frac{|\ln \e|}{2\ln( R_2/R_1)}\) then \(E_\e(u_\e)\leq \pi d_\e^2\ln \frac{R_2}{R_1}\).
\item[2)] If \(d_\e/|\ln \e| \to \alpha\) with \(\alpha  > \frac{1}{2\ln ( R_2/R_1)}\) then \(E_\e(u_\e)\leq \pi |\ln \e|^2(\alpha-\frac{1}{4\ln (R_2/R_1)})(1+o(1))\).
\item[3)] If \( |\ln \e| \ll d_\e \ll \sqrt{|\ln \e|}/\e\) then \(E_\e(u_\e) \leq \pi d_\e |\ln \e|(1+o(1))\).
\end{itemize}
\end{proposition}

\begin{proof} Let us rewrite the upper bound 
 \eqref{eq:upperbound} in Proposition \ref{prop:2nd_upper_bound} in terms of \(N_\e=d_\e-\gamma_\e\):
\begin{multline}\label{eq:upper_intermediate}
E_\e(u_\e)\leq \pi (d_\e-N_\e)^2 \ln \frac{R_2}{R_1}+\pi N_\e|\ln \e| \\
 +\pi N_\e |\ln N_\e|+CN_\e+\pi (d_\e-N_\e)N_\e \ln \frac{1}{1-\frac{\pi}{2N_\e}}+O(d_\e\e+d_\e^3\e^2). 
\end{multline} When \(N_\e\) is large, this upper bound can be rewritten as
\begin{equation}\label{eqfNN}\pi d_\e^2 \ln \frac{R_2}{R_1} +\pi |\ln \e|^2 f\left (\frac {N_\e}{|\ln \e|} \right ) +O(|\ln \e| \ln |\ln \e|)\end{equation} where $f$ is defined by
\begin{equation}\label{eqfN}
f(x)=x^2\ln \frac{R_2}{R_1}+x\left (1-2\frac{d_\e}{|\ln \e|}\ln \frac{R_2}{R_1}\right).
\end{equation}From this function, we see that the critical value $|\ln \e|/2\ln (R_2/R_1)$ plays a role.

In the case 1), the best choice to minimize $f$ is $N_\e=0$, and the upper bound  is in fact a direct consequence of Lemma \ref{lem:1st_upper_bound}. 

In case 2), we choose  \[N_\e= \lfloor d_\e -\frac{|\ln \e|}{2\ln (R_2/R_1)}\rfloor\] to minimize $f$ and find from \eqref{eqfNN}
\begin{align*}
E_\e(u_\e) & \leq \pi \left( d_\e- \lfloor d_\e-\frac{|\ln \e|}{2 \ln (R_2/R_1)} \rfloor \right)^2\ln \frac{R_2}{R_1}+\pi \lfloor d_\e-\frac{|\ln \e|}{2\ln (R_2/R_1)}\rfloor  \\
& \quad +O(|\ln \e| \ln |\ln \e|) \\
& \leq \pi \frac{|\ln \e|^2}{4\ln(R_2/R_1)}|+\pi \alpha |\ln \e|^2-\frac{|\ln \e|^2}{2\ln (R_2/R_1)}+O(|\ln \e| \ln |\ln \e|) \\
 & \leq \pi (\alpha-\frac{1}{4 \ln (R_2/R_1)})|\ln \e|^2+O(|\ln \e| \ln |\ln \e|).
\end{align*}
Item 3) is obtained by taking $\gamma_\e=0$, \(N_\e=d_\e\) in the upper bound of Proposition \ref{prop:2nd_upper_bound} and by verifying that \(d_\e^3\e^2=o(d_\e|\ln \e|)\) when \(d_\e \ll \sqrt{|\ln \e|}/\e\).
\end{proof}

As a side remark, we observe that if we assume that $d_\e$ is of order $  \frac{|\ln \e|}{2 \ln (R_2/R_1)}+\omega_\e \) with \(1\ll \omega_\e\ll |\ln \e|$, then the number of vortices \(N_\e\) minimizing the upper bound in Proposition \ref{prop:2nd_upper_bound} is \(N_\e=\omega_\e(1+o(1))\). Indeed, by using \eqref{eq:upper_intermediate} we find after computations that 
\begin{align}\label{eq:upper_intermediate22}
E_\e(u_\e)\leq \frac{\pi|\ln \e|^2}{4\ln (R_2/R_1)}+\pi\omega_\e|\ln \e|+\pi(\omega_\e-N_\e)^2 \ln \frac{R_2}{R_1}+\pi N_\e |\ln N_\e| \\
+CN_\e+\pi (d_\e-N_\e)N_\e \ln \frac{1}{1-\frac{\pi}{2N_\e}}+O(d_\e\e+d_\e^3\e^2). \nonumber
\end{align}
The minimization of the right-hand side with respect to \(N_\e\) leads to \(N_\e=\omega_\e(1+o(1))\). We will not use this fact in the sequel.

\section{Extending to apply the vortex-balls construction}\label{sec:tools}

 As explained before, the main difficulty of the problem is that  vortices concentrate on the boundary \(\p B_{R_2}\). Thus we cannot directly use the classical tools for the GL energy to find a lower bound. We solve this problem by first extending the minimizers in a slightly bigger annulus and then using the vortex-balls construction and the Jacobian estimate in the extended domain.

Let \(a>0\), we define

\begin{equation}\label{eq:extension}
\tilde{u}_\e=\begin{cases}
u_\e & \text{ in } A \\
e^{id_\e\theta} & \text{ in } B_{R_2+a}\setminus B_{R_2}.
\end{cases}
\end{equation}

We set   \(A_a := B_{R_2+a}\setminus \overline{B}_{R_1}\) and  \(A_a^\e:=\{ x\in A_a; \dist (x, \p A_a)>\e\}\). We observe that the extended current-vector and vorticity measure of \(\tilde{u}_\e\) are given by
\begin{equation*}
\tilde{j}_\e :=\begin{cases}
u_\e \wedge \nabla u_\e  & \text{ in } A \\
\frac{d_\e}{r}\vec{e_\theta}  & \text{ in } B_{R_2+a}\setminus B_{R_2}
\end{cases} \end{equation*}
\begin{equation*}
\tilde{\mu}_\e =\curl \tilde{j}_\e =\begin{cases}
2\p_x u_\e \wedge \p_y u_\e  & \text{ in } A \\
0 & \text{ in } B_{R_2+a}\setminus B_{R_2}.
\end{cases}
\end{equation*}

 We first recall the two classical tools in the GL theory that we will use in the extended domain. We start with the vortex-balls construction originally due to Jerrard \cite{Jerrard_1999} and  Sandier \cite{Sandier_1998}.

\begin{theorem}[Theorem 4.1 in \cite{Sandier_Serfaty_2007}]\label{th:vortex_ball_construction}
For all \(\delta \in (0,1)\) there exists \(\e_0(\delta)>0\) such that, for all \(\e<\e_0\), if \(u\in H^1(A_a,\mathbb{C})\) satisfies \(E_\e(u,A_a) \leq  \e^{\delta-1}\) then the following properties hold.  For all \(1>r>C\e^{\delta/2}\), with \(C\) some universal constant there exists a finite collection of disjoint closed balls \(\mathcal{B}=\{B_i\}_{i\in I}\) such that
\begin{itemize}
\item[1)] \(r(\mathcal{B}):=\sum_i r(B_i)=r\)
\item[2)] If \(V=A_a^\e\cap \bigcup_{i\in I} B_i\) then
\[\{x\in A_a^\e; ||u(x)|-1| \geq \e^{\delta/4}\}\subset V.\]
\item[3)] Writing \(n_i=\deg (u,\p B_i)\) if \(B_i \subset A_a^\e\) and \(n_i=0\) otherwise,
\[ E_\e(u_\e,V)\geq \pi D\left( \ln \frac{r}{D \e}-C \right),\]
with \(D=\sum_{i\in I} |n_i|\) is assumed not to vanish and \(C\) a universal constant.
\item[4)] \(D\leq C \frac{E_\e(u_\e,A)}{|\ln \e|}\), with \(C\) another universal constant.
\end{itemize}
\end{theorem}

The second tool that we will use is the Jacobian estimate originally due to Jerrard-Soner \cite{Jerrard_Soner_2002_Calvar} and one of its consequences.

\begin{theorem}[ Theorem 6.1 in \cite{Sandier_Serfaty_2007}]\label{th:jacobian_estimate}
Let \(u\in H^1(A_a,\mathbb{C})\), let \(\mathcal{B}=\{B_i\}_{i\in I}\) be a finite collection of disjoint closed balls and let \(\e>0\) be such that
\[ \{x\in A_a^\e; |u(x)-1|\geq 1/2\} \subset \cup_i B_i \]

Then \[ \| 2\p_x u \wedge \p_y u -2\pi \sum_{i\in I, B_i\subset A_a^\e} n_i \delta_{a_i}\|_{(C_c^{0,1}(A_a))^*}\leq C \max (r,\e)(1+E_\e(u_\e,A_a))\]
where \(a_i\) is the center of \(B_i\), \(r=\sum_ir(B_i)\), \(n_i=\deg (u/|u|, \p B_i)\) and \(C\) is a universal constant.
\end{theorem}

\begin{theorem}[ Theorem 6.2 in \cite{Sandier_Serfaty_2007}]\label{th:jacobian_estimate_2}
Assume \(\delta \in (0,1)\) and \(\e<\e_0(\delta)\), where \(\e_0(\delta)\) is given by Theorem \ref{th:vortex_ball_construction}. Assume \(E_\e(u)<\e^{\delta-1}\) and let \( \mathcal{B}=\{B(a_i,r_i)\}_{i\in I}\) denote a collection of balls given by Theorem \ref{th:vortex_ball_construction} for some \(\e^{\frac{\delta}{2}}<r<1\). We let 
\[ \nu=2\pi \sum_{i \in I; B_i \subset A_a^\e} n_i \delta_{a_i}\]
where \(n_i=\deg(u, \p B_i)\) and \(\mu=2\p_xu \wedge \p_yu\). Then, writing \(M=E_\e(u)\), we have
\begin{equation*}
\| \mu-\nu\|_{(\C^{0,1}_c(A_a))^*}\leq Cr (M+1) \quad \text{ and } \| \nu \|_{\C(A_a)^*}\leq C\frac{M}{\delta |\ln \e|},
\end{equation*}
where \(C\) is a universal constant. Moreover for any \(\beta \in (0,1)\) there exists a constant \(c_\beta\) depending on \(\beta\) and \(A_a\) and \(\e_0(\delta,\beta)\) such that if \(\e<\e_0\), then
\begin{equation*}
\|\mu\|_{(\C_c^{0,\beta}(A_a))^*}\leq C_\beta \frac{M+1}{\delta |\ln \e|}, \quad \|\mu-\nu\|_{(\C_c^{0,\beta}(A_a))^*}\leq c_\beta r^\beta (M+1).
\end{equation*}
\end{theorem}

\section{The case \(d_\e\) of order \(|\ln \e|\): proof of Theorem \ref{th:main1}}\label{sec:d_order_lneps}

\subsection{Lower bound of the energy}

We start by deriving a lower bound on the energy of minimizers. Note that the lower bound we will obtain is valid for general domains as long as the GL energy is of order \(|\ln \e|^2\). 

\begin{proposition}\label{prop:first_lower_bound}
Let \( (u_\e)_\e\) be a family of minimizers of \(E_\e\) in \(\I_\e\) and assume that \(d_\e/|\ln \e| \to \alpha\). Let \(a>0\) and consider the extensions \(\tilde{u}_\e\) given by \eqref{eq:extension}.
Then, up to a subsequence \(\e\to 0\),
\begin{equation*}
\tilde{j}_\e \rightharpoonup \tilde{j} \text{ in } L^2(A_a) \quad \text{ and }  \tilde{\mu}_\e \rightharpoonup \tilde{\mu} \text{ in } (\C_c^{0,\gamma}(A_a))^* \text{ for any } \gamma\in (0,1).
\end{equation*}
Furthermore, \(\curl \tilde{j}=\tilde{\mu}\) in \(A_a\), \(\dive \tilde{j}=0\) in \(A\) and \(\supp \tilde{\mu} \subset \overline{B_{R_2}}\setminus \overline{B_{R_1}}\). We have
\begin{equation}\label{eq:lower_bound00}
\liminf_{\e \to 0}\frac{E_\e (u_\e)}{|\ln \e|^2}\geq \frac{\|\curl \tilde{j}\|}{2}+\frac{1}{2}\int_{A} |\tilde{j}|^2
\end{equation}
\end{proposition}
Here \(\|\mu\|\) denotes the total variation of the measure \(\mu\), i.e., \(\|\mu\|=|\mu|(A_a)\).

\begin{proof} We follow the proof of \cite[Theorem 7.1, Item 1)]{Sandier_Serfaty_2007}. We first observe that, when \(d_\e/|\ln \e| \to \alpha\), then the GL energy of minimizers is of order \(|\ln \e|^2\) thanks to Lemma \ref{lem:1st_upper_bound}.  We can write that
\begin{equation}\label{eq:sum_energy}
E_\e(u_\e,A)=E_\e(\tilde{u}_\e, A_a)-\pi d_\e^2 \ln \frac{R_2+a}{R_2}.
\end{equation}
We apply the vortex-balls construction of Theorem \ref{th:vortex_ball_construction} to \(\tilde{u}_\e\) to find a finite collection of balls \(B_i^\e=B(x_{i,\e},r_{i,\e})\) for \(i \in I_\e\) such that
\begin{equation*}
E_\e(\tilde{u}_\e,A_a)\geq \pi D_\e\left( \ln \frac{r_\e}{D_\e \e}-C \right)+\frac12 \int_{A_a \setminus \left( \cup B_i^\e \cap A_a^\e \right)} |\nabla \tilde{u}_\e|^2,
\end{equation*}
where 
\[r_\e=\sum_{i\in I_\e} r_{i,\e}, \quad A_a^\e=\{x\in A_a; \dist (x,\p A_a)>\e\},  \quad D_\e=\sum_{i\in I_\e} |n_{i,\e}|,\] with \(n_{i,\e}=\deg (u_\e/|u_\e|,\p B_i^\e)\) if \(B_i^\e\subset A_a^\e\) and \(n_{i,\e}=0\) otherwise  and \(C>0\) is a universal constant. We can choose \(r_\e\) which goes to zero but with \(|\ln r_\e| \ll |\ln \e|\) since, from the energy upper bound, \(r_\e\) can be chosen so that \(1>r_\e>C\e^{\delta/2}\) for any \(\delta \in (0,1)\). From Theorem \ref{th:vortex_ball_construction} we also have
\begin{equation}\label{eq:bound_number_vortices}
|D_\e| \leq C \frac{E_\e(\tilde{u}_\e,A_a)}{\delta|\ln \e|}, \quad \forall \delta \in (0,1).
\end{equation}
In particular we find that \( |D_\e|\leq C|\ln \e|\) and we obtain the lower bound
\begin{equation}\label{eq:lower_bound_1}
E_\e(\tilde{u}_\e,A_a)\geq \pi D_\e |\ln \e| -o(|\ln \e|^2)+\frac12 \int_{A_a\setminus \left( \cup B_i^\e \cap A_a^\e \right)} |\nabla \tilde{u}_\e|^2.
\end{equation}
Now we use that \(|\nabla \tilde{u}_\e|^2\geq |\tilde{u}_\e\wedge \nabla \tilde{u}_\e|^2=|\tilde{j}_\e|^2\) in \(A_a\).

We have that \(\frac{\tilde{j}_\e}{|\ln \e|}\) is bounded in \(L^2(A_a)\), thus, up to a subsequence it converges weakly in \(L^2\) towards some \(\tilde{j}\in L^2(A_a)\). We then take a sequence \(\e_n \to 0\) and we set \(\mathcal{A}_N=\bigcup_{n\geq N} V_{\e_n}\) where
\begin{equation*}
V_{\e_n}=\left( \bigcup_{i\in I_{\e_n}} B_i^{\e_n} \right) \cap \{x\in A_a; \dist (x,\p A_a)>\e_n\} .
\end{equation*}
Up to a subsequence, we have that \( |\mathcal{A}_N|\to 0\) as \(N \to +\infty\). Indeed, recall that \(r_{\e}=\sum_{i \in I_\e} r_{i,\e}\) and observe that \(\sum_{i \in I_\e} |r_{i,\e}|^2\leq r_\e^2\). Now, since \(r_\e\to 0\), up to a subsequence we have that the series \(\sum_{n\geq 1} r_{\e_n}^2\to 0\) converges and this implies that \(|\mathcal{A}_N|\to 0\) as \(N\to +\infty\).
But for every \(N\) fixed, we see that
\begin{align*}
\liminf_{n \to +\infty} \int_{A_a \setminus V_{\e_n}}\frac{|\tilde{j}_{\e_n}|^2}{|\ln \e_n|^2} & \geq \liminf_{n \to +\infty} \int_{A_a\setminus \mathcal{A}_N}\frac{|\tilde{j}_{\e_n}|^2}{|\ln \e_n|^2} \geq \int_{A_a \setminus \mathcal{A}_N} |\tilde{j}|^2.
\end{align*}
We can then pass to the limit \(N \to +\infty\) to find
\begin{equation}\label{eq:interaction_part}
\liminf_{n \to +\infty} \int_{A_a \setminus V_{\e_n}}\frac{|\tilde{j}_{\e_n}|^2}{|\ln \e_n|^2} \geq \int_{A_a} |\tilde{j}|^2.
\end{equation}

Now, since \(j_\e\cdot \tau=j_\e\cdot \vec{e_\theta}=\frac{d_\e}{r}\) on \(\p B_{R_2}\), a computation in the sense of distribution shows that
\begin{equation*}
\curl \tilde{j}_\e =\begin{cases}
\curl j_\e &\text{ in } A \\
0 & \text{ in } B_{R_2+a} \setminus \overline{B}_{R_2}
\end{cases}=2\p_x \tilde{u}_\e\wedge \p_y \tilde{u}_\e.
\end{equation*}
From \eqref{eq:bound_number_vortices} we deduce that \(\left( 2\pi \sum_{i\in I_{\e_n}} n_{i,\e_n}\delta_{x_{i,\e_n}}\right) /|\ln {\e_n}|\) is a  bounded sequence of measures and thus, up to a subsequence, it converges to some limit. From the Jacobian estimate of Theorem \ref{th:jacobian_estimate} this limit is also the limit of \(\tilde{\mu}_{\e_n}/|\ln \e_n|=\curl \tilde{j}_{\e_n}\). Since  \(\tilde{\mu}_{\e_n}/|\ln \e_n| \rightharpoonup \curl \tilde{j}\) in the sense of distributions we find that
\begin{equation*}
\frac{2\pi \sum_{i\in I_{\e_n}} d_{i,\e_n}\delta_{x_{i,\e_n}}}{|\ln {\e_n}|}\rightharpoonup \tilde{\mu}:=\curl \tilde{j} \text{ in } \mathcal{M}(A_a).
\end{equation*}

Hence,
\begin{equation}\label{eq:vortices-part}
\liminf_{\e_n \to 0} \frac{\pi |D_{\e_n}| |\ln {\e_n}|}{|\ln \e_n|^2}\geq \frac{\|\tilde{\mu}\|(A_a)}{2}.
\end{equation}
Note that Theorem \ref{th:jacobian_estimate_2} implies that \(\tilde{\mu}_\e\) converges to \(\tilde{\mu}\) in \((\C_c^{0,\gamma}(A_a))^*\) for any \(\gamma\in (0,1)\).

By using \eqref{eq:lower_bound_1},\eqref{eq:interaction_part} and \eqref{eq:vortices-part} we find that
\begin{equation}\label{eq:lower_bound0}
\liminf_{\e_n\to 0}\frac{E_{\e_n}(\tilde{u}_{\e_n},A_a)}{|\ln \e_n|^2}\geq  \frac{\|\tilde{\mu}\|(A_a)}{2}+ \frac12 \int_{A_a} |\tilde{j}|^2
\end{equation}
with \(\curl \tilde{j}=\tilde{\mu}\) in \(A_a\) and \(\supp\tilde{\mu} \subset \overline{B}_{R_2}\setminus \overline{B}_{R_1}\). This last property comes from the fact that \(\curl \tilde{j}_\e=0\) in \(B_{R_2+a}\setminus \overline{B}_{R_2}\) for every \(\e\). Now we obtain \eqref{eq:lower_bound00} by using \eqref{eq:lower_bound0}, \eqref{eq:sum_energy} and by observing that \(\int_{A_a\setminus A}|\tilde{j}|^2=\pi \alpha^2\ln (\frac{R_2+a}{R_2})\).
\end{proof}

\subsection{Minimization of the function \(F\)}\label{sec:min_F}

The result of the previous section leads us to study the minimization of the functional
\begin{equation}\label{eq:functionnelle_F}
F^a(j) =\frac{\|\curl j\|(A_a)}{2}+\frac{1}{2}\int_{A} |j|^2
\end{equation}
defined on
\begin{multline}\label{eq:min_space}
\mathcal{J}^a:=\Bigl\{j\in L^2(A_a); \dive (j)=0 \text{ in } A, \curl j \in \M(A_a), \\
\supp \curl j \subset \overline{B}_{R_2}\setminus B_{R_1}, j=\frac{\alpha}{r}\vec{e_\theta} \text{ in } B_{R_2+a}\setminus B_{R_2}\Bigr\}.
\end{multline}
Actually, for \(j\in \J^a\) we rewrite \(F^a(j)\) as 
\begin{equation}\label{eq:functionnelle_F_bis}
F^a(j) =\frac{\|\curl j\|(A_a)}{2}+\frac{1}{2}\int_{A_a} |j|^2-\pi \alpha^2 \ln \frac{R_2+a}{R_2}.
\end{equation}
A first remark is that \(\mathcal{J}^a\) is a vectorial space and \(F^a\) is strictly convex on \(\J^a\). We can also check that, since the total variation and the \(L^2\) norm are weakly lower semi-continuous for the corresponding topologies, \(F^a\) is weakly lower semi-continuous. An application of the direct method in the calculus of variations yields that \(F^a\) admits a unique minimizer in \(\mathcal{J}^a\). We are able to describe more precisely this minimizer.

\begin{proposition}\label{prop:jmin_radial}
Let \(j_0^a\) be the unique minimizer of \(F^a\) in \(\J^a\) then \(j_0^a=f^a(r) \vec{e_\theta}\) for some \(f^a\in L^2((R_1,R_2+a))\) and \(\vec{e_\theta}=(-\sin \theta, \cos \theta)\).
\end{proposition}

\begin{proof}
Let \(\vec{e_r}=(\cos \theta,\sin \theta)\).  We consider the action of the group of rotations on elements of \(\mathcal{J}^a\) defined by \((R,j)\in SO(2)\times \mathcal{J}^a \mapsto R j(R^{-1}\cdot)\). The subset \(\J^a\cap \{j \in L^2(A_a); j=g(r)\vec{e_r}+f(r) \vec{e_\theta}\}\) is the set of invariants of \(\J^a\) by this action. It can be checked that the functional \(F^a\) is invariant by this action. Furthermore we can write \( F^a=F_0^a+F_1^a\) with
\begin{equation*}
F_0^a(j)=\frac{1}{2}\int_{A_a} |j|^2, \quad F_1^a(j)=\frac{\|\curl j\|(A_a)}{2}.
\end{equation*}
Since \(F_0^a\) is \(\C^1\) on \(\mathcal{J}^a\) and \(F_1^a\) is a convex, proper, lower semi-continuous function which is invariant by the action of \(SO(2)\), we can apply Proposition 3 and Theorem 4 in  \cite{Squassina_2011} to deduce that the principle of symmetric criticality  of Palais holds true in our situation. Hence we obtain that the minimizer of \(F^a\) in \(\J^a \cap \{j \in L^2(A_a); j=g(r)e_r+f(r) \vec{e_\theta}\}\) is a critical point of \(F^a\) in \(\J^a\). But by strict convexity, this critical point is the unique minimizer of \(F^a\) in \(\J^a\).
However from the condition \(\dive j_0^a= \frac{1}{r}(rg(r))'=0\) in \(A\), we find that \(g(r)=C/r\) for some constant \(C\) in \(A\). But \(\curl \frac{C}{r}\vec{e_r} =0\). Hence we can take \(C=0\), since having \(C\neq 0\) only increases the energy \(F^a\).
\end{proof}

The rest of this section is devoted to find the function \(f^a\) such that \(j_0^a=f^a(r)\vec{e_\theta}\) in \(A\). We aim to show

\begin{proposition}\label{prop:explicit_j}
Let \(j_0^a\) be the unique minimizer of \(F^a\) in \(\J^a\) then \(j_0^a=f^a(r) \vec{e_\theta}\) in \(A\) with \(|f^a(r)|=\frac{\lambda}{r}\) for some constant \(\lambda\) in  \(\R^+\) and \(\curl j_0^a\) can be supported only on some circles in \(\overline{A}\).
\end{proposition}

Thanks to Proposition \ref{prop:jmin_radial}, the problem of describing the minimizer \(j_0^a\) of \(F^a\) in \(\J^a\) reduces to a 1D problem. This is due to fact that we are working in a  symmetric domain with rotationally invariant boundary condition. We will obtain Proposition \ref{prop:explicit_j} by proving  that \(j_0^a=\nabla^\perp h^a\) with \(h^a\) which is stationary harmonic in \(A\). Here by ``stationary harmonic'' we mean that it is a critical point of the Dirichlet energy for inner variations or that the stress-energy tensor is divergence-free, cf.\ e.g.\ \cite[Theorem 1.4.15]{Helein_2002}. In the radial case, this stationarity condition means that \((r^2|h'|^2)'=0\).

First to each \(j\in \J^a\) of the form \(j=f(r)\vec{e_\theta}\) we associate a unique \(h\in H^1(A_a)\) defined  by
\begin{equation}\label{eq:potential}
\left\{
\begin{array}{rcll}
\Delta h&=&\mu & \text{ in } A_a \\
\p_\nu h&=& \frac{\alpha}{R_2+a} & \text{ on } \p B_{R_2+a} \\
h&=&0  & \text{ on } \p B_{R_1},
\end{array}
\right.
\end{equation}
 where \(\mu=\curl j\). Indeed, if \(j\in \J^a\) can be written \(f(r)\vec{e_\theta}\) then we can write \(j=\nabla^\perp h\) for \(h\) radial and \(h=\int_{R_1}^r f(s) ds\). We then see that \( \curl j=\mu=\Delta h\) and, since \(j=(\alpha/r) \vec{e_\theta}\) on \(\p B_{R_2+a}\) we obtain that \(j\cdot \vec{e_\theta}=\nabla^\perp h\cdot \vec{e_\theta}=\p_\nu h=\alpha/(R_2+a)\) on \( \p B_{R_2+a}\). This shows that the current vectors \(j\) in \(\J^a\) which can be written \(j=f(r)\vec{e_\theta}\) are uniquely determined by \(\mu= \curl j\in \mathcal{M}(A_a)\) or by \(h\) the unique solution to \eqref{eq:potential}. We set
\begin{equation}\label{eq:Irad}
 \J_{rad}= \{j\in \J^a; \exists f\in L^2( (R_1,R_2+a)), \ j=f(r) \vec{e_\theta}\}.
 \end{equation}

 We next show that we can rewrite the energy \(F^a(j)\) in terms of \(h\) or in terms of \(\mu\). To this end we introduce the following Green function:

\begin{equation}\label{eq:Green}
\left\{
\begin{array}{rcll}
-\Delta_x G(x,y) &=& \delta_y & \text{ in } A_a, \\
\p_{\nu} G(x,y)& =& 0 & \text{ on } \p B_{R_2+a}, \\
G(x,y)& =& 0  & \text{ on } \p B_{R_1}.
\end{array}
\right.
\end{equation}
 This Green function satisfies the usual following properties.

\begin{proposition}\label{prop:properties_Green_function}
Let \(G\) be the Green function defined by \eqref{eq:Green} then
\begin{itemize}
\item[1)] \(G(x,y)=G(y,x)\) for all \(x\neq y\) in \(A_a\).
\item[2)] \(G(x,y)=-\frac{1}{2\pi}\ln |x-y| +\varphi(x,y)\) where \(\varphi\) is in \(\C^\infty(A_a\times A_a)\).
\item[3)] Actually we have that \( \varphi\) is in \(\C^\infty(\overline{A}_a\times A_a)\) and in \(\C^\infty(A_a\times \overline{A}_a)\).
\end{itemize}
\end{proposition}

\begin{proof}
\begin{itemize}
\item[1)] The symmetry of the Green function is classical and the proof of Theorem 13 p.35 in \cite{Evans_2010} adapts.
\item[2)] If we write \(\varphi(x,y)=G(x,y)-\frac{1}{2\pi}\ln |x-y|\) then we can see that \(\Delta_x\varphi =0\) in \(A_a\). Furthermore, for every \(y\) in \(A_a\), we observe that \(\varphi(\cdot,y)\) is in \(\C^\infty(\p B_{R_1})\) and \(\p_{\nu} \varphi (\cdot,y)\) is also in \(\C^\infty(\p B_{R_2+a})\). Elliptic regularity theory, see e.g.\ \cite{Nardi_2014}, provides that \(x \mapsto \varphi(x,y)\) is \(\C^\infty(\overline{A}_a)\) for every \(y\) in \(A_a\). But by the symmetry property of item 1) we also obtain that \(y \mapsto \varphi(x,y)\) is \(\C^\infty(A_a)\). Since the two partial maps of \(\varphi\) are in \(\C^\infty(A_a)\) we deduce that \(\varphi\) is in \(\C^\infty(A_a\times A_a)\).
\item[3)] The third points is proved thanks to the explicit formula of the Green function given in Appendix \ref{sec:AppendixA}.
\end{itemize}
\end{proof}

\begin{proposition}\label{prop:energy_measure_green}
Let \(j\in \J_{rad}\) and let \(\mu:=\curl j\). Then \(F^a(j)=H^a(\mu)\) where
\begin{multline}\label{eq:energy_measure}
H^a(\mu)=\frac{\|\mu\|(A_a)}{2}+ \frac12 \iint_{A_a\times A_a} G(x,y) d\mu(x)d\mu(y) \\
-\frac{\alpha}{2(R_2+a)} \Bigl( \int_{\p B_{R_2+a}} \int_{A_a}G(x,y) d\mu(x) d\mathcal{H}^1(y) + \int_{\p B_{R_2+a}} \int_{A_a}G(x,y) d\mu(y) d\mathcal{H}^1(x)\Bigr) \\
+\frac{\alpha^2}{2(R_2+a)}\int_{\p B_{R_2+a}}\int_{B_{R_2+a}} G(x,y) d \mathcal{H}^1(x) d \mathcal{H}^1(y)-\pi \alpha^2\ln \frac{R_2+a}{R_2}.
\end{multline}
\end{proposition}

\begin{proof}
Assume first that \(\mu \in \C^\infty(A_a)\) then, from elliptic regularity theory, see e.g.\ \cite[Theorem 6.19]{Gilbarg_Trudinger_2001} and \cite{Nardi_2014}, the solution to \eqref{eq:potential} is smooth up to the boundary and can be represented with the Green function \eqref{eq:Green}. More precisely, see e.g.\ p.34 in \cite{Evans_2010},
\begin{equation*}
h(x)=\int_{\p A_a} G(x,y) \p_\nu h(y) d\mathcal{H}^1(y)-\int_{\p A_a} \p_\nu G(x,y) h(y) d\mathcal{H}^1(y) -\int_{A_a} G(x,y) \mu(y) dy.
\end{equation*}
Thus,
\begin{equation}\label{eq:representation_h}
h(x)= \frac{\alpha}{R_2+a}\int_{\p B_{R_2}} G(x,y) d\mathcal{H}^1(y)-\int_{A_a} G(x,y) \mu(y) d y.
\end{equation}

We can actually see that \eqref{eq:representation_h} continues to hold when \(\mu \in \M(A_a)\) such that there exists \(j\in \J^a\cap \J_{\text{rad}} \) with \(\mu=\curl j\). Note that such a measure \(\mu\) satisfies that \(\mu \in H^{-1}(A_a)\) since it can be written as the divergence of an \(L^2\) function.  We can also find \(\mu_n\in \C^\infty(A_a)\) such that \(\mu_n \rightharpoonup \mu\) in \(\M(A_a)\), see e.g.\ \cite[Proposition 2.7]{Ponce_2016}. By elliptic regularity, cf.\ e.g.\ \cite[Proposition 5.1]{Ponce_2016}, \(h_n\) solution of \eqref{eq:potential} with \(\mu=\mu_n\), converges to \(h\) solution de \eqref{eq:potential} in \(W^{1,p}_{\text{loc}}(A_a)\) for any \(1\leq p<2\).  Let us show that
\begin{equation*}
\int_{A_a} G(x,y) d \mu_n(y) \rightarrow \int_{A_a} G(x,y) d \mu(y) \text{ for } \mu-\text{a.e. } x\in A_a.
\end{equation*}
In order to do that we write \(G(x,y)=-\frac{1}{2\pi}\ln |x-y|+\varphi (x,y)\) with \(\varphi\) as in Proposition \ref{prop:properties_Green_function}. Let \(\eta_\delta\in \C^\infty_c(\R^2)\) such that \(\eta_\delta\equiv 1\) in \(A_a\setminus \{x \in A_a: \dist (x,\p B_{R_1})<\delta \text{ and } \dist( x, \p B_{R_2+a})<\delta\}\) and \(\eta_\delta =0\) on \(\p B_{R_1}\) and on \( \p B_{R_2+a}\) with \(\eta_\delta \to 1\) a.e.\ in \(A_a\) as \(\delta\to 0\). We then have
\begin{align*}
\left| \int_{A_a}\varphi(x,y) d (\mu_n-\mu)(y) \right|\leq \left| \int_{A_a} \eta_\delta(y) \varphi(x,y) d (\mu_n-\mu)(y) \right| \\
+ \left|\int_{A_a} (1-\eta_\delta)(y) \varphi(x,y) d (\mu_n-\mu)(y) \right|.
\end{align*}
The approximation \(\mu_n\) being given by the convolution of \(\mu\) with a regularizing kernel it satisfies \(|\mu_n-\mu| \leq 2 |\mu|\). From Proposition \ref{prop:properties_Green_function} we know that, for all \(x\in A_a\), the map \(y \mapsto \varphi(x,y)\in \C^\infty(\overline{A}_a)\). Hence
\begin{align*}
\left|\int_{A_a} (1-\eta_\delta(y)) \varphi(x,y) d (\mu_n-\mu)(y) \right|&\leq C_x \int_{A_a}(1-\eta_\delta) d |\mu|(y)  \xrightarrow[\delta \to 0]{} 0
\end{align*}
by dominated convergence theorem. On the other hand, since for any \(\delta>0\) and  for \(x\in A_a\) we have that  \(y \mapsto \eta_\delta(y) \varphi(x,y)\in \C^\infty_c(A_a)\), we obtain
\[\int_{A_a} \eta_\delta(y) \varphi(x,y) d (\mu_n-\mu)(y)\xrightarrow[n \to +\infty]{} 0.\]
Hence we deduce that 
\begin{equation}\label{eq:cv_varphi}
\int_{A_a}  \varphi(x,y) d (\mu_n-\mu)(y)\xrightarrow[n \to +\infty]{} 0 \text{ for every } x\in A_a.
\end{equation}

Now we can apply \cite[Theorem 1]{Brezis_Browder_1979} to deduce that if \(\mu \in H^{-1}(A_a)\) and \(\mu \geq 0\) then
\(\left|\int_{A_a}\ln |x-y| d\mu (y)\right| <+\infty\) for \(\mu-\)a.e.\ \(x\in A_a\).
Indeed let \(G_0\) be the Green function in \(A_a\) with Dirichlet boundary condition, i.e., the solution to
\begin{equation*}
\left\{
\begin{array}{rcll}
-\Delta_xG_0(\cdot,y) &= & \delta_y & \text{ in } A_a \\
G_0(\cdot,y) &=& 0 & \text{ on } \p A_a.
\end{array}
\right.
\end{equation*} 
As in Proposition \ref{prop:properties_Green_function}  we have that \(G_0(x,y)=-\frac{1}{2\pi}\ln |x-y|+\varphi_0(x,y)\) for \(x,y\in A_a\) and \(\varphi_0\in \C^\infty(\overline{A}_a\times A)\cap \C^\infty(A_a\times \overline{A}_a)\). We note that \(G_0\geq 0\) from the maximum principle \cite[Proposition 6.1]{Ponce_2016}. We define \(h_0(x)=\int_{A_a}G_0(x,y) d \mu(y)\) for \(x\in A_a\). Then, thanks to Fubini's Theorem we can check that \(h_0 \in L^1(A_a)\). We can also check that
\begin{equation*}
\left\{
\begin{array}{rcll}
-\Delta h_0 &= & \mu & \text{ in } A_a \\
h_0 &=& 0 & \text{ on } \p A_a.
\end{array}
\right.
\end{equation*}
Furthermore, since \(\mu \in H^{-1}(A_a)\) we have that \(h_0\in H^1_0(A_a)\) and \( h_0\geq 0\) thanks to the maximum principle. Hence \( \langle \mu , h_0\rangle_{\langle H^{-1},H^1_0\rangle}\geq 0\). We can then apply Theorem 1 of \cite{Brezis_Browder_1979} to deduce that  \(h_0\in L^1(A_a,\mu)\). But it can be shown that \(x\mapsto \int_{A_a}\varphi_0(x,y) d\mu(y)\in \C(\overline{A}_a)\) hence we obtain that 
\begin{equation}\label{eq:Browder1}
\left| \int_{A_a}\ln |x-y| d \mu(y) \right| <+\infty \text{ for } \mu-\text{ a.e. } x\in A_a.
\end{equation} 

Now, for \(\mu \in \mathcal{K}:=\{\mu \in \M(A_a); \exists j\in \J^a\cap \J_{\text{rad}} \text{ with } \mu=\curl j\}\) we can write that for \(\mu\)-a.e.\ \(x\in A_a\) and for \(\delta>0\),
\begin{align*}
\left| \int_{A_a}\ln |x-y| d(\mu_n-\mu)(y) \right| \leq \left| \int_{A_a\setminus B_\delta(x)} \ln |x-y| d(\mu_n-\mu)(y)\right|  \\
\quad + \left| \int_{B_\delta(x)}\ln |x-y| d(\mu_n-\mu)(y)\right|.
\end{align*}
The first term in the right-hand side converges to zero when \(n \to +\infty\) because \(\mu_n \rightharpoonup \mu\) in \(\M(A_a)\) and \(y \mapsto \ln |x-y|\in \C(\overline{A_a\setminus B_\delta(x)})\) (the proof is the same as for \eqref{eq:cv_varphi}). For the second term we use that \(\mu \geq 0\) and hence  \(y \mapsto \ln |x-y| \in L^1(A_a,|\mu|)\) from \eqref{eq:Browder1} and the dominated convergence theorem  to obtain that 
\begin{align*}
\sup_{n\in \mathbb{N}} \left|\int_{B_\delta(x)} \ln |x-y| d(\mu_n-\mu)(y) \right|&\leq 2\left|\int_{B_\delta(x)}\ln |x-y| d|\mu|(y) \right| \xrightarrow[\delta \to 0]{} 0.
\end{align*}
This proves that, if \(h \) is a solution to \eqref{eq:potential} with \(\mu\in \mathcal{K}\) then
\begin{equation}\label{eq:Browder2}
 h(x)= \frac{\alpha}{R_2+a}\int_{\p B_{R_2}} G(x,y) d\mathcal{H}^1(y)-\int_{A_a} G(x,y) d \mu(y) \text{ for }\mu-\text{a.e. } x\in A_a.
\end{equation}
Now let \(\mu_1, \mu_2 \in \mathcal{K}\) with \(\mu_1,\mu_2\geq 0\) we call \(h^{\mu_1}\) the solution to \eqref{eq:potential} with \(\mu=\mu_1\) and \(h^{\mu_2}\) the solution to \eqref{eq:potential} with \(\mu=\mu_2\). Let \(\tilde{\eta}_\delta\in \C^\infty_c(\R^2)\) be such that \(\tilde{\eta}_\delta\equiv 1\) on \(\{x\in A_a; \dist(x,\p B_{R_2+a})>\delta\}\) and \(\tilde{\eta}_\delta=0\) on \(\p B_{R_2+a}\) with \(\tilde{\eta}_\delta \to 1\) as \(\delta \to 0\) a.e. We can apply again Theorem 1 in \cite{Brezis_Browder_1979} to obtain that 
\begin{align*}
\langle \mu_1, \tilde{\eta}_\delta h^{\mu_2} \rangle_{\langle {H^{-1},H^1_0}\rangle} &= \int_{A_a}\tilde{\eta}_\delta h^{\mu_2} d \mu_1 =\int_{A_a} h^{\mu_1} d\mu_2
\end{align*}
for \(\delta>0\) small enough since \(\supp \mu_1 \subset \overline{B}_{R_2}\setminus B_{R_1}\). But on the other hand
\begin{align*}
\langle \mu_1, \tilde{\eta}_\delta h^{\mu_2} \rangle_{\langle {H^{-1},H^1_0}\rangle} &= \langle \Delta h^{\mu_1}, \tilde{\eta}_\delta h^{\mu_2} \rangle_{\langle {H^{-1},H^1_0}\rangle}\\
&=-\int_{A_a}\nabla h^{\mu_1} \cdot \nabla (\tilde{\eta}_\delta h^{\mu_2}) \\
&=-\int_{A_a}\nabla h^{\mu_1} \cdot \nabla h^{\mu_2} \tilde{\eta}_\delta+\int_{A_a} \nabla \tilde{\eta}_\delta\cdot \nabla h^{\mu_2} h^{\mu_1}.
\end{align*}
By letting \(\tilde{\eta}_\delta \to 1\) in \(A_a\) and by using that \(h_{\mu_1}=0\) on \(\p B_{R_1}\) and \(\p_\nu h_{\mu_2}\) is a constant on \(\p B_{R_2+a}\) (hence better than just in \(H^{-1/2}(\p B_{R_2+a}\))),  we obtain
\begin{equation}\label{eq:*}
\int_{A_a}\nabla h^{\mu_1} \cdot \nabla h^{\mu_2}=-\int_{A_a}h^{\mu_2} d \mu_1 +\int_{ \p A_a}\p_\nu h^{\mu_1} h^{\mu_2}.
\end{equation}
In particular, if \(\mu \geq 0\), by taking \(\nu=\mu\) in \eqref{eq:*} and by using \eqref{eq:Browder2} we obtain \eqref{eq:energy_measure}. Now if \(\mu\) is not necessarily non-negative then we write \(\mu=\mu^+-\mu^-\) with \(\mu^{\pm}\geq 0\). An important observation is that, since \(\mu=\curl j\) for some \(j \in \J^a\cap \J_{\text{rad}}\) then \(\mu\) can be written as a product \(\mu=\mu_r \otimes \mathcal{L}^1_{\lfloor (0,2\pi)}\) where \(\mathcal{L}^1\) is the Lebesgue measure and \(\mu_r\) is a 1D measure depending only on the radius, in polar coordinates. We can do the same for \(\mu^+\) and \(\mu^-\). Moreover in 1D we do have the embedding \(\M( (R_1,R_2+a)) \subset H^{-1}((R_1,R_2+a))\). Since we work in an annulus and then, \(r\geq R_1>0\) we deduce that \(\mu^+\) and \(\mu^-\) are in \(H^{-1}(A_a)\). Then we can write the relation \eqref{eq:*} for \((h^{\mu^+},h^{\mu^+}) \), \((h^{\mu^+},h^{\mu^-})\) and \((h^{\mu^-},h^{\mu^-})\) and combining these relations we obtain \eqref{eq:energy_measure}.
\end{proof}

\begin{proposition}\label{prop:Inner_variations_mu}
A minimizer \(\mu_0^a\) of \(H^a\) in \(\mathcal{K}^a=\{\mu \in \M(A_a); \exists j\in \J^a\cap \J_{\text{rad}} \text{ with } \mu=\curl j\}\) satisfies the following relation for every \(X\in \C^\infty_c(A,\R^2)\) such that \(X\) is radial:\footnote{we say that \(X\) is radial if it can be written as \(X=\chi(r) \vec e_r\) for some function \(\chi\).}
\begin{multline}\label{eq:inner_equal_zero}
\iint_{A_a\times A_a} \left(\nabla _x G(x,y)\cdot X(x) + \nabla _y G(x,y)\right) \cdot X(y) \, d\mu_0^a(x)d\mu_0^a(y) \\
-\frac{\alpha}{(R_2+a)} \Bigl( \int_{\p B_{R_2+a}} \int_{A_a} \nabla_x G(x,y)\cdot X(x) \, d\mu_0^a(x) d\mathcal{H}^1(y) \\
+ \int_{\p B_{R_2+a}} \int_{A_a} \nabla _y G(x,y)\cdot X(y) d\mu_0^a(y) d\mathcal{H}^1(x)\Bigr)=0.
\end{multline}
\end{proposition}

\begin{proof}
Let \(\mu\) be a measure in \(\mathcal{K}^a\). We use variations of the form \(\mu_t:=(\text{ Id} +tX)_\# \mu\) with \(X\) a radial vector field and \(X\in \C^\infty_c(A,\R^2)\). We notice that, since \(X\) is radial, \(\mu_t\) is indeed in \(\mathcal{K}^a\).

From the definition of the total variation we see that \(\| \mu_t\|(\overline{A})=\| \mu\|(\overline{A})\). Now
by definition of the push-forward of a measure
\begin{align*}
H^a(\mu_t)&= \frac{\|\mu\|(\overline{A})}{2}+\frac12 \iint_{A_a\times A_a} G(x+tX(x),y+tX(y)) d\mu(x) d\mu(y) \\
& +\frac{\alpha^2}{(R_2+a)^2} \int_{\p B_{R_2+a}}\int_{\p B_{R_2+a}} G(x,y) d\mathcal{H}^1(x) d \mathcal{H}^1(y) \\
&- \frac{\alpha}{R_2+a} \int_{\p B_{R_2}}\int_{A_a} G(x+tX(x),y) d\mu(x) d \mathcal{H}^1(y) \\
&- \frac{\alpha}{R_2+a} \int_{\p B_{R_2}} \int_{A_a} G(x,y+tX(y)) d\mu(y) d \mathcal{H}^1(x)\\
&-\pi \alpha^2\ln \frac{R_2+a}{R_2}.
\end{align*}
An application of the theorem of derivation under the integral sign shows that \[\frac{d}{dt}\Big|_{t=0 }H^a((\text{ Id} +tX)_\# \mu))\] exists and is given by the left-hand side of \eqref{eq:inner_equal_zero}. Since \(\mu_0^a\) minimizes \(H^a\) we must have \(\frac{d}{dt}\Big|_{t=0 }H^a((\text{ Id} +tX)_\# \mu_0^a))=0\).  To dominate the derivative, to apply the theorem of derivation under the integral sign,  we use that, for every \(y\in \p B_{R_2+a}\) and for every \(x\in A_a\) we have \(|\nabla_xG(x,y)\cdot X(x)|\leq C\) since \(G\in \C^\infty(A_a\times \overline{A_a})\). In the same way for every \(x\in \p B_{R_2+a}\) and every \(y \in A_a\) we can see that \(|\nabla_yG(x,y)\cdot X(y)|\leq C\) because \(G\in \C^\infty(\overline{A_a}\times A_a)\). We can also observe, thanks to  Lemma \ref{lem:domination} below, that for every \(x,y \in A_a\) we have 
\[ | \nabla_xG(x,y)\cdot X(x)+\nabla_yG(x,y)\cdot X(y)|\leq C.\]
\end{proof}

\begin{lemma}\label{lem:domination}
Let \(X\in \C^\infty_c(A,\R^2)\), let \(G\) be the Green function \eqref{eq:Green} then there exist \(Q_i\in \C(\R^2\setminus \{0\})\) for \(i=1,\dots,3\) with \(Q_i\) which is bounded and independent of \(X\) and \(\psi_{i,X}\in \C_0(\R^2\times \R^2)\) for \(i=1,\dots,3\) such that 
\begin{multline}\label{eq:majoration_derivee}
\nabla _x G(x,y)\cdot X(x) + \nabla _y G(x,y)\cdot X(y) = Q_1(x-y)\psi_{1,X}(x,y)+ Q_2(x-y)\psi_{2,X}(x,y)\\+ Q_3(x-y)\psi_{3,X}(x,y)\
+\nabla _x \varphi(x,y)\cdot X(x) + \nabla _y \varphi(x,y)\cdot X(y),
\end{multline}
where \(\varphi\) is defined in Proposition \ref{prop:properties_Green_function}.
We observe that from Item 3) of Proposition \ref{prop:properties_Green_function} we have that \((x,y)\in A\times A \mapsto \nabla _x \varphi(x,y)\cdot X(x) + \nabla _y \varphi(x,y)\cdot X(y)\) is continuous and bounded on \(A\times A\).
\end{lemma}

\begin{proof}
From Proposition \ref{prop:properties_Green_function} we can write that \(G(x,y)=-\frac{1}{2\pi}\ln |x-y|+\varphi(x,y)\) where \(\varphi\) is in \(\C^\infty(A_a\times A_a)\). Thus
\begin{align*}
\nabla _x G(x,y)\cdot X(x) + \nabla _y G(x,y)\cdot X(y) & =\frac{-1}{2\pi}\left(\frac{x-y}{|x-y|^2}\cdot (X(x)-X(y))\right) \\
& \phantom{aaaa} +\nabla _x \varphi(x,y)\cdot X(x) + \nabla _y \varphi(x,y)\cdot X(y)
\end{align*}

By using Taylor's expansion theorem and by extending \(X\) by zero outside \(A\), we can write that for every \(x,y\in A\)
\begin{align*}
X(x)-X(y)=\int_0^1 DX(sx+(1-s)y)(x-y)\, ds .
\end{align*}
Thus
\begin{align*}
&\frac{x-y}{|x-y|^2}\cdot (X(x)-X(y)) =\sum_{i,k=1}^2\int_0^1 \p_kX^i(sx+(1-s)y)ds (x-y)_k\frac{(x-y)_i}{|x-y|^2} \\
&= \frac{(x_1-y_1)^2}{|x-y|^2} \int_0^1 \p_1X^1 (sx+(1-s)y) \, ds \\
& +\frac{(x_2-y_2)^2}{|x-y|^2} \int_0^1 \p_2X^2 (sx+(1-s)y) \, ds  \\
&+ \frac{(x_1-y_1)(x_2-y_2)}{|x-y|^2}\left(\int_0^1 \left[(\p_1X^2+\p_2X^1)(sx+(1-s)y) \right] \, ds \right).
\end{align*}

Thanks to the theorems of continuity and limits under the integral sign we can see that for \(i,k\in \{1,2\}\) we have that \((x,y)\mapsto \int_0^1 \p_i X^k (sx+(1-s)y)ds\) is in \(\C_0(\R^2\times \R^2)\) (the set of continuous functions vanishing at infinity). We obtain \eqref{eq:majoration_derivee} with
\[ Q_1(z)=\frac{z_1^2}{|z|^2}, \quad Q_2(z)=\frac{z_2^2}{|z|^2}, \quad Q_3(z)= \frac{z_1z_2}{|z|^2}.\]

\end{proof}

\begin{proposition}\label{prop:stress_energy_tensor_divfree}
Let \(\mu\) be in \(\mathcal{K}^a\), let \(h^\mu\) be the potential associated to \(\mu\) via \eqref{eq:potential}. For every \(X\in \C^\infty_c(A,\R^2)\) with \(X\) radial, we have

\begin{multline}\label{eq:derivative_inner}
\iint_{A_a\times A_a} \left(\nabla_x G(x,y)\cdot X(x) + \nabla_y G(x,y)  \cdot X(y)\right) d\mu(x)d\mu(y) \\
-\frac{\alpha}{(R_2+a)} \Bigl[  \int_{\p B_{R_2+a}} \int_{A_a}\nabla_x G(x,y)\cdot X(x) d\mu(x) d\mathcal{H}^1(y) \\
+ \int_{\p B_{R_2+a}} \int_{A_a} \nabla_y G(x,y)\cdot X(y) d\mu(y) d\mathcal{H}^1(x)\Bigr] =\int_{A_a} [h^\mu,h^\mu]:DX
\end{multline}
where \( [h^\mu,h^\mu]_{ij}=2\p_ih^\mu\p_jh^\mu-|\nabla h^\mu|^2\delta_{ij}\) for \(i,j=1,2\), \(\delta_{ij}\) is the Kronecker symbol and \(P:Q=\tr (PQ^T)\) denotes the inner product of matrices \(P,Q\) in \(M_2(\R)\).
\end{proposition}

\begin{proof}
Again we start with the case where \(\mu\) is in \(\mathcal{K}^a\cap\C^\infty(A_a)\). In that case the Green representation formula shows that for every \(x\in A_a\),
\begin{equation*}
\nabla h^\mu(x)=  \frac{\alpha}{R_2+a}\int_{\p B_{R_2}} \nabla_xG(x,y) d\mathcal{H}^1(y)-\int_{A_a} \nabla_xG(x,y) \mu(y) d y.
\end{equation*}
Thus, the left-hand side of \eqref{eq:derivative_inner} can be seen to be equal to
\begin{equation*}
\int_{A_a} \nabla h^\mu(x)\cdot X(x) d\mu(x) +\int_{A_a}\nabla h^\mu(y)\cdot X(y) d\mu(y).
\end{equation*}
On the other hand a direct computation shows that \(\dive [h^\mu,h^\mu]=2\Delta h^\mu \nabla h^\mu\). This shows that the left-hand side of \eqref{eq:derivative_inner} is equal to
\begin{equation*}
\int_{A_a} \dive [h^\mu,h^\mu]\cdot X(x) dx=\int_{A_a} [h^\mu,h^\mu]:DX(x) dx.
\end{equation*}

In the general case, we proceed again by approximation and we follow \cite[Proposition 6.4]{Peszek_Rodiac_2025}.We first extend \(\mu\) by zero outside \(A\) and we call \(\tilde{\mu}\) the new measure. We take a sequence \((\tilde{\mu}_n)_n \subset \C^\infty (\R^2)\) such that
\begin{itemize}
\item  \(\supp \tilde{\mu}_n \Subset A_a\) for \(n \) large enough,
\item \(\tilde{\mu}_n\in H^{-1}(\R^2)\) and \(\tilde{\mu}_n\) converges strongly to \(\tilde{\mu}\) in \(H^{-1}(\R^2)\). Note that this implies that \(\tilde{\mu}_n\) converges strongly to \(\mu\) in \(H^{-1}(A_a)\),
\item \(\tilde{\mu}_n \rightharpoonup \tilde{\mu}\) in the sense of measures \(\mathcal{M}(\R^2)\) and \(|\mu_n| \rightharpoonup |\mu|\) in the sense of measures in \(\mathcal{M}(\R^2)\).
\end{itemize}
The construction of such an approximating sequence is done by convolution with a regularizing kernel, cf.\ \cite[Theorem 2.2]{Ambrosio_Fusco_Pallara_2000}. By the strong convergence of \(\tilde{\mu}_n\) in \(H^{-1}(A_a)\) it follows that the associated potentials \(h^{\tilde{\mu}_n}\) converge strongly in \(H^1(A_a)\) to \(h^{\tilde{\mu}}\). Thus \([h^{\tilde{\mu}_n},h^{\tilde{\mu}_n}]\xrightarrow [n\to +\infty]{L^1(A_a)} [h^{\tilde{\mu}},h^{\tilde{\mu}}]\) and
\begin{equation*}
\int_{A_a} [h^{\tilde{\mu}_n},h^{\tilde{\mu}_n}]:DX(x) dx \xrightarrow[n \to +\infty]{}  \int_{A_a} [h^{\tilde{\mu}},h^{\tilde{\mu}}]:DX(x) dx.
\end{equation*}

Now we note that for every \(y\) in \(\p B_{R_2+a}\) the map \(x \mapsto \nabla_x G(x,y)\cdot X(x)\) is in \(\C_c(A_a)\). By definition of the convergence in the sense of measures, for every \(y\) in \(\p B_{R_2+a}\)
\[ \int_{A_a}\nabla_xG(x,y)\cdot X(x) d\tilde{\mu}_n (x) \xrightarrow[]{n \to +\infty} \int_{A_a}\nabla G_x(x,y) \cdot X(x) d \tilde{\mu}(x). \]
From Proposition \ref{prop:properties_Green_function} we know that \( (x,y) \mapsto \nabla_xG(x,y)\cdot X(x)\) is bounded on \(\overline{A}\times \p B_{R_2+a}\). Hence
\begin{align*}
\left| \int_{A_a} \nabla_x G(x,y)\cdot X(x) d \tilde{\mu}_n(x) \right| &\leq \| \nabla_xG(x,y)\cdot X(x)\|_{L^\infty(\overline{A}\times \p B_{R_2+a}) } \|\tilde{\mu}_n\|(A_a) \leq C
\end{align*}
since the convergence in the sense of measures implies that the total variation of \( \tilde{\mu}_n\) is uniformly bounded in \(n\). Thus the dominated convergence theorem shows that
\begin{multline*}
 \int_{\p B_{R_2+a}} \int_{A_a} \nabla_x G(x,y)\cdot X(x) d\tilde{\mu}_n(x) d\mathcal{H}^1(y) \\
 \xrightarrow[n \to +\infty]{}   \int_{\p B_{R_2+a}} \int_{A_a} \nabla_x G(x,y)\cdot X(x) d\tilde{\mu}(x) d\mathcal{H}^1(y).
\end{multline*}

To pass to the limit in the term
\begin{multline*}
\iint_{A_a\times A_a} \left(\nabla_x G(x,y)\cdot X(x) + \nabla_y G(x,y)  \cdot X(y) \right)d\tilde{\mu}_n(x)d\tilde{\mu}_n(y)\\
=\iint_{\R^2 \times \R^2} \left(\nabla_x G(x,y)\cdot X(x) + \nabla_y G(x,y)  \cdot X(y) \right)d\tilde{\mu}_n(x)d\tilde{\mu}_n(y)
\end{multline*}
 the main ingredient is Lemma \ref{lem:domination}.  Once we know that \(\left(\nabla_x G(x,y)\cdot X(x) + \nabla_y G(x,y) \right) \cdot X(y)\) can be decomposed as 
\begin{multline*}
Q_1(x-y)\psi_{1,X}(x,y)+ Q_2(x-y)\psi_{2,X}(x,y)\\+ Q_3(x-y)\psi_{3,X}(x,y)\
+\nabla _x \varphi(x,y)\cdot X(x) + \nabla _y \varphi(x,y)\cdot X(y) \text{ in } \R^2\times \R^2
\end{multline*}
 with \(Q_i\) bounded for \(i=1,2,3\) and \(\psi_{i,X}\) continuous and vanishing at infinity, we can use the convergence \(|\mu_n| \rightharpoonup |\mu|\) in the sense of measures in \(\mathcal{M}(\R^2)\) and Proposition \ref{prop:Appendixb} in Appendix to conclude. Note that we use that the mass of the \(\tilde{\mu}_n\) is uniformly bounded to be able to use that testing against functions with compact in \(\R^2\times \R^2\) is equivalent to testing with functions vanishing at infinity.
\end{proof}

\begin{proof}(proof of Proposition \ref{prop:explicit_j})
From Proposition \ref{prop:Inner_variations_mu} and Proposition \ref{prop:stress_energy_tensor_divfree} we find that
\begin{equation}\label{eq:stress_div_free}
\int_{A_a} [h^{\mu_0^a},h^{\mu_0^a}]: DX \, dx=0, 
\end{equation}
for all \(X\in \C^\infty_c(A)\), with \(X=X(r)\vec{e_r}\), (with some abuse of notation) and where \(h^{\mu_0^a}\) is such that \(j_0^a=\nabla^\perp h^{\mu_0^a}\). Since we know from Proposition \ref{prop:jmin_radial} that \(j_0^a=f^a(r) \vec{e_\theta}\) in \(A_a\) we have that \(h^{\mu_0^a}\) is radial. In the basis \((\vec{e_r},\vec{e_\theta})\) we have
\begin{align*}
[h^{\mu_0^a},h^{\mu_0^a}]= 2\nabla h^{\mu_0^a} \otimes \nabla h^{\mu_0^a} -|\nabla h^{\mu_0^a}|^2 \text{Id}=\begin{pmatrix}
-|(h^{\mu_0^a})'|^2 &0 \\
0 & |(h^{\mu_0^a})'|^2.
\end{pmatrix}
\end{align*}
In the same basis we can write \( DX =\begin{pmatrix}
X'(r) &0 \\
0 & \frac{X(r)}{r}
\end{pmatrix}\). Thus \eqref{eq:stress_div_free} leads to 
\begin{align*}
\int_{R_1}^{R_2} |(h^{\mu_0^a})'|^2 \left(X'(r)-\frac{X(r)}{r}\right) r \, dr = \int_{R_1}^{R_2} r^2|(h^{\mu})'|^2 \left(\frac{X(r)}{r}\right)' \, dr =0.
\end{align*}

This translates into \(r^2|{(h^{\mu_0^a})}^{'}|^2=\lambda\) in \((R_1,R_2)\) for some constant \(\lambda\in \R^+\). We thus obtain that \(|j_0^a|=\frac{\lambda}{r}\vec{e_\theta}\). Furthermore, since \(h^{\mu_0^a}\) satisfies \eqref{eq:stress_div_free} and \(\Delta h^{\mu_0^a}=\mu_0^a\) is a Radon measure, we can apply the results \cite[Theorems 1.3, 1.4, 1.6]{Rodiac_2016} to deduce that \(\mu_0^a\) can only be supported on some union of smooth curves. Since \(h^{\mu_0^a}\) is radial, these curves can only be circles in \(A\). Hence \(\mu_0^a\) can be only supported on some circles in \(A\) and on \(\p B_{R_2}\).
\end{proof}

\begin{proposition}
Let \(j_0^a\) be the minimizer of \(F^a\)  in \(\mathcal{J}^a\). Then there exists a constant \(\lambda \in \R^+\) such that \(j_0^a=\frac{\lambda}{r}\vec{e_\theta}\). Moreover \(\mu=\curl j_0^a=0\) in \(A\). In other words the limiting vorticity \(\mu \in \mathcal{M}(B_{R_2+a}\setminus \overline{B}_{R_1})\) is supported on \(\p B_{R_2}\).
\end{proposition}

\begin{proof}
We use Proposition \ref{prop:explicit_j} and we observe that among all vector fields of the form \(j=\chi(|x|)\frac{\lambda}{r}\vec{e_\theta}\) with \(\chi:A \rightarrow \{\pm 1\}\) the one which has the least \(F\)-energy is \(j_0^a=\frac{\lambda}{r}\vec{e_\theta}\) in \(A\). Indeed, having a constant sign for \(\chi\) decreases the measure part and does not change the Dirichlet part.
\end{proof}

\subsection{Proof of the main results}
We are now ready to conclude the proof of the main result Theorem \ref{th:main1} and of Item 2) of Theorem \ref{th:uniqueness_d_small}.

\begin{proof}(proof of Item 2) Theorem \ref{th:uniqueness_d_small} and of Theorem  \ref{th:main1})
From the lower bound and the minimization of the energy \(F^a\) in Section \ref{sec:min_F} we have that
\begin{equation*}
\liminf_{\e \to 0} \frac{E_\e(u_\e)}{|\ln \e|^2} \geq  F^a(\tilde{j}) \geq F^a(j_0^a),
\end{equation*}
where \(j_0^a=\frac{\lambda}{r}\vec{e_\theta}\) is the minimizer of \(F^a\) and \(\tilde{j}\) is the weak limit in \(L^2(A_a)\) of \(\tilde{j}_\e/|\ln \e|\). Let us compute \(\curl j_0^a\) in \(A_a\). Let \(\varphi \in \C^\infty_c(B_{R_2+a}\setminus \overline{B}_{R_1})\),
\begin{align}
\langle \curl j_0^a, \varphi \rangle &= -\int_{B_{R_2+a}\setminus \overline{B}_{R_1}} j_0^a\cdot \nabla^\perp \varphi \nonumber \\
&=-\int_{B_{R_2+a}\setminus B_{R_2}}\frac{\alpha}{r}\p_r \varphi -\int_{B_{R_2}\setminus B_{R_1}} \frac{\lambda}{r}\p_r\varphi 
=\int_{\p B_{R_2}} \varphi \frac{(\alpha-\lambda)}{R_2} \, d \mathcal{H}^1. \label{eq:calcul_curl}
\end{align}
Hence \(\curl j_0^a=\frac{(\alpha-\lambda)}{R_2} \mathcal{H}^1_{\lfloor \p B_{R_2}}\). Thus we deduce that
\[F^a(j_0^a)=\pi |\alpha-\lambda|+\pi \lambda^2 \ln \frac{R_2}{R_1}.\]
We observe that this quantity is independent of \(a>0\). We can now minimize this as a function of \(\lambda\) and actually this is similar to the function $f$ defined by \eqref{eqfN} in the proof of Proposition \ref{cor:upper_bound_final}. We find that if \(\alpha \leq \frac{1}{2\ln \frac{R_2}{R_1}}\) then the minimum is attained for \(\lambda=\alpha\) and if \(\alpha > \frac{1}{2\ln \frac{R_2}{R_1}}\) it is attained for \( \lambda=\frac{1}{2\ln \frac{R_2}{R_1}}\). We thus have
\begin{align*}
\liminf_{\e \to 0} \frac{E_\e(u_\e)}{|\ln \e|^2}& \geq  F^a(\tilde{j}) \geq F^a(j_0^a) \geq \pi |\alpha-\lambda|+\pi \lambda^2 \ln \frac{R_2}{R_1},\\
& \geq \begin{cases}
\pi \alpha^2 \ln \frac{R_2}{R_1}  & \text{ if } \alpha \leq \frac{1}{2 \ln (R_2/R_1)} \\
 \pi \left(\alpha-\frac{1}{4\ln \frac{R_2}{R_1}}\right) & \text{ if } \alpha > \frac{1}{2 \ln (R_2/R_1)}.
\end{cases}  
\end{align*}
But thanks to the upper bounds in Lemma \ref{lem:1st_upper_bound} and Proposition \ref{cor:upper_bound_final} we also have that 
\[  \liminf_{\e \to 0} \frac{E_\e(u_\e)}{|\ln \e|^2}\leq \begin{cases}
\pi \alpha^2 \ln \frac{R_2}{R_1}  & \text{ if } \alpha \leq \frac{1}{2 \ln (R_2/R_1)} \\
 \pi \left(\alpha-\frac{1}{4\ln \frac{R_2}{R_1}}\right)  & \text{ if } \alpha > \frac{1}{2 \ln (R_2/R_1)}.
\end{cases} .\]
This entails, by the strict convexity of \(j \in \J^a \mapsto F^a(j)\) that \( \tilde{j}=j_0^a\) and we know explicitly \(j_0^a\) according to the value of \(\alpha\). In particular if \(\alpha \geq 1/(2\ln (R_2/R_1)\) then \(j_0^a=\frac{1}{2\ln (R_2/R_1) r}\vec{e_\theta}\) in \(A\) and \(\tilde{\mu}_\e/|\ln \e|\rightharpoonup \curl j_0^a =(\alpha-\frac{1}{2\ln (R_2/R_1)})\mathcal{H}^1_{\lfloor \p B_{R_2}}\) in \(\mathcal{M}(A_a)\). Hence we obtain that \(\tilde{\mu}_\e/|\ln \e|\rightharpoonup \mathcal{H}^1_{\lfloor \p B_{R_2}}\) in \(\M(\overline{B}_{R_2}\setminus \overline{B}_{R_1})\). This completes the proof of Item 2) Theorem \ref{th:uniqueness_d_small} and of Theorem \ref{th:main1}.
\end{proof}

\section{The case \(d_\e < \alpha_c|\ln \e|\): proof of Theorem \ref{th:uniqueness_d_small}}\label{sec:d_small}
This section is devoted to the proof of Theorem \ref{th:uniqueness_d_small}. We follow closely the proofs of Theorem 1.2 and Theorem 1.3 in \cite{Aftalion_Jerrard_RoyoLetelier_2011}. We first need some information on the radial solution of the GL equation in our annulus.

\begin{lemma}\label{lem:estimates_eta}
Let \(\eta_\e\) be a minimizer of the energy \eqref{eq:energy_eta} among real-valued functions in \(H^1((R_1,R_2))\) satisfying \(\eta(R_1)=\eta(R_2)=1\). Then \(\eta_\e\) is in \(\C^\infty([R_1,R_2])\) and satisfies
\begin{equation}\label{eq:estimate_eta}
\|\eta_\e-1\|_{L^\infty}\leq C(d_\e\e)^2, \quad \quad \| \nabla \eta_\e\|_{L^\infty}\leq C,
\end{equation}
for some constant \(C>0\). Furthermore there is uniqueness of the minimizer.
\end{lemma}
This lemma can be proved by using the techniques in \cite{BBH_93,Aftalion_2006}.

In Ginzburg-Landau theory, to prove uniqueness of minimizers in some special situations  it is customary to use a splitting of the energy which takes advantage of the  fact that we work with complex-valued function. This was originally devised in \cite{Lassoued_Mironescu_1999}. 

\begin{lemma}\label{lem:decomposition_energy} 
Let \(u\) be in \(\I_\e\), thanks to Lemma \ref{lem:estimates_eta}, for \(\e>0\) small enough, \(\eta_\e\) does not vanish in \(A\). We can thus write \(u=\eta_\e e^{id_\e\theta} v\) for some \(v\in H^1(A,\mathbb{C})\) with \(|v|=1\) on \(\p B_{R_1}\) and \(v=1\) on \(\p B_{R_2}\). Then
\begin{equation}\label{eq:splittinh}
E_\e(u)=E_\e(\eta_\e e^{id_\e \theta})+G_\e(v),
\end{equation} 
where 
\begin{equation*}
G_\e(v)=\frac12 \int_A \eta_\e^2 |\nabla v|^2+\frac{1}{4\e^2}\int_A \eta_\e^4 (1-|v|^2)^2 +d_\e\int_A \frac{\eta_\e^2}{r}\vec{e_\theta}\cdot (v\wedge \nabla v).
\end{equation*}
\end{lemma}

\begin{proof}
We let \(w_\e:=\eta_\e e^{id_\e \theta}\). We drop the subscript \(\e\) for simplicity, a computation shows that
\begin{align*}
E_\e(w v)&=\frac12 \int_A \left(|v|^2 |\nabla w|^2 + |w|^2 |\nabla v|^2 +2v\nabla w\cdot w \nabla v\right) +\frac{1}{4\e^2}\int_A (1-|w v|^2)^2 \\
&= \frac12 \int_A \left( |v|^2 \left[|\nabla \eta|^2+\eta^2 \frac{d_\e^2}{r^2}\right]+\eta^2|\nabla v|^2\right)\\
& \quad + \int_A \eta \eta' v\cdot \p_rv + \int_A \frac{d_\e \eta^2}{r^2} v\wedge \p_\theta v +\frac{1}{4\e^2}\int_A (1-|wv|^2)^2.
\end{align*}
Thus
\begin{align*}
E_\e(wv)-E_\e(w) &=\frac12 \int_A (|v|^2-1) (|\nabla \eta|^2+\eta^2 \frac{d_\e^2}{r^2})+\frac12 \int_A \eta^2|\nabla v|^2 \\
& \quad +\int_A \eta \eta' v\cdot \p_rv+\int_A \frac{d_\e \eta^2}{r^2} v\wedge \p_\theta v \\
& \quad \quad +\frac{1}{4\e^2}\int_A [(1-|wv|^2)^2)-(1-|w|^2)^2].
\end{align*}
We can check that \((1-|wv|^2)^2)-(1-|w|^2)^2=2|w|^2(1-|v|^2)-|w|^4(1-|v|^4)\).
Now we use the equation satisfied by \(\eta_\e\). It reads
\begin{equation}\label{eq:EL_eta}
\left\{
\begin{array}{rcll}
-\Delta \eta +\frac{d_\e^2  }{r^2}\eta &=& \frac{1}{\e^2}\eta(1-\eta^2) &\text{ in } A, \\
\eta&=&1 & \text{ on } \p A.
\end{array}
\right.
\end{equation}
We multiply \eqref{eq:EL_eta} by \(\eta (1-|v|)\), which  vanishes on \(\p A\), and integrate by parts to find that
\begin{multline}\label{*}
\int_A |\nabla \eta|^2(|v|^2-1)+2\int_A \eta \nabla \eta \cdot v \nabla v+\int_A \eta^2 (|v|^2-1)\frac{d_\e^2}{r^2} \\
= \frac{1}{\e^2}\int_A \eta^2(1-\eta^2)(|v|^2-1).
\end{multline}
By using \eqref{*} we obtain that 
\begin{align*}
E_\e(wv)-E_\e(w)&= \frac12 \int_A \eta^2|\nabla v|^2+\frac{1}{2\e^2}\eta^2(1-\eta^2)(|v|^2-1)+\frac{1}{2\e^2}\int_A \eta^2(1-|v|^2) \\
& \quad -\frac{1}{4\e^2}\int_A \eta^4 (1-|v|^4)+\int_A \frac{d_\e \eta^2}{r}\vec{e_\theta}\cdot (v\wedge \nabla v).
\end{align*}
To conclude we observe that 
\begin{align*}
\frac{1}{2\e^2}\eta^2(1-\eta^2)(|v|^2-1)+\frac{1}{2\e^2}\eta^2(1-|v|^2)-\frac{1}{4\e^2} \eta^4 (1-|v|^4) \\
= -\frac{2\eta^4}{4\e^2}(|v|^2-1)-\frac{1}{4\e^2}\eta^4 (1-|v|^4) \\
=\frac{\eta^4}{4\e^2}(1-2|v|^2+|v|^4)=\frac{\eta^4}{4\e^2}(1-|v|^2)^2.
\end{align*}
\end{proof}

Now our goal is to show that if \(u_\e=w_\e v\) is  a minimizer of \(E_\e\) in \(\I_\e\) then \(E_\e(v)\) is small (like \(\e^\beta\) for some \(\beta\)) in \(A\). With this information and with the help of a gradient bound on \(v\) of the type \(|\nabla v|< C/\e\) then we can conclude that \(v\) converges uniformly to \(1\). However we were able to obtain the desired gradient bound only far from the boundary \(\p B_{R_1}\). We will need to work in the extended annulus \(B_{R_2+a}\setminus B_{R_1}\) as at the beginning of Section \ref{sec:d_order_lneps}. Note that, if we extend \(u_\e\) and \(w_\e\) by \(e^{id_\e \theta}\) in \(B_{R_2+a}\setminus \overline{B}_{R_2}\) then \(v\) is automatically extended by \(1\) and the extensions satisfy the decomposition 
\begin{equation*}
E_\e(u,\tilde{A})=E_\e(w_\e,\tilde{A})+G_\e(v,\tilde{A}).
\end{equation*}
We define 
\begin{equation}\label{def:xi}
\xi_\e:=\int_{R_1}^r \frac{\eta_\e^2(s)}{s} ds \quad \text{ for } r\in (R_1,R_2)
\end{equation}
so that \(\nabla^\perp \xi_\e=\eta_\e^2(r) \vec{e_\theta}/r\) in \(A\). We also extend \(\xi_\e\) to \((R_2,R_2+a)\) by setting \(\xi_\e(r)=\xi_\e(R_2)\) in  \((R_2,R_2+a)\). Thanks to Lemma \ref{lem:estimates_eta} we can prove that 
\begin{equation}\label{eq:prop_xi}
\|\xi_\e\|_{L^\infty(\A)}\leq \ln (R_2/R_1)+O((d_\e\e)^2), \quad \|\nabla \xi_\e\|_{L^\infty(\tilde{A})}\leq C. 
\end{equation}
We can rewrite the term whose sign is not controlled in \(G_\e\) thanks to this function \(\xi_\e\) and an integration by parts:

\begin{align*}
\int_{A} \frac{\eta_\e^2}{r}\vec{e_\theta}\cdot v\wedge \nabla v &=\int_{A} \nabla^\perp \xi_\e \cdot v\wedge \nabla v \\
&= -\int_{A}\xi_\e \curl (v\wedge \nabla v)+\int_{\p \A} \xi_\e v\wedge \p_\tau v \\
&= -\int_{A}\xi_\e \mu_v
\end{align*}
where we defined \(\mu_v:= \curl (v\wedge \nabla v)=2\p_xv\wedge\p_yv\) and where we used that \(\xi_\e=0\) on \(\p B_{R_1}\) and \(v\wedge \p_\tau v=0\) on \(\p B_{R_2}\) since \(v=1\) there. Since \(v=1\) in \(B_{R_2+a}\setminus \overline{B}_{R_2}\) we can also say that
\begin{equation*}
\int_{\A} \frac{\eta_\e^2}{r}\vec{e_\theta}\cdot v\wedge \nabla v =-\int_{\A}\xi_\e \mu_v.
\end{equation*}
We can thus rewrite 
\begin{align*}
G_\e(v,\A)&=\frac12 \int_{\A}\eta^2|\nabla v|^2-d_\e \int_{\A}\xi_\e \mu_v+\frac{1}{4\e^2}\eta^4 (1-|v|^2)^2 \\
&=\frac12 \int_{\A}\eta^2\left[|\nabla v|^2-\frac{2\xi_\e d_\e}{\eta_\e}\mu_v\right]+\frac{1}{4\e^2}\int_{\A}\eta^4(1-|v|^2)^2.
\end{align*}
Now we take a non-vanishing smooth function \(a_\e=o_\e(1)\) such that \(d_\e^2/a_\e \ll \e\) and \(\chi_\e\in \C^\infty_c(\A)\) such that 
\begin{multline*}
\chi_\e \equiv 1 \text{ in } B_{R_2+a-\frac{2a_\e}{d_\e}}\setminus \overline{B}_{R_1+\frac{2a_\e}{d_\e}}, \quad \chi_\e\equiv 0 \text{ in } B_{R_1+\frac{2a_\e}{d_\e}}\setminus \overline{B}_{R_1}, \\
 \chi_\e\equiv 0 \text{ in }  B_{R_2+a}\setminus \overline{B}_{R_2+a-\frac{2a_\e}{d_\e}}, \quad |\nabla \chi_\e|\leq C\frac{d_\e}{a_\e}.
\end{multline*}
We decompose
\begin{equation*}
G_\e(v,\A)= A_1-A_2+B,
\end{equation*}
where
\begin{align*}
A_1&=\int_{\A} \chi_\e \left[\frac{\eta_\e^2}{2}|\nabla v|^2+\frac{\eta_\e^4}{4\e^2}(|v|^2-1)^2 \right], \\
A_2&= d_\e \int_{\A} \chi_\e \xi_\e \mu_v, \\
B&=\int_{\A} (1-\chi_\e ) \left[ \frac{\eta_\e^2}{2}(|\nabla v|^2-2 d_\e f_\e \mu_v) +\frac{\eta_\e^4}{4\e^2}(|v|^2-1)^2 \right].
\end{align*}
We used the following notation \(f_\e:=\xi_\e/\eta_\e^2\). Because of our choice of \(\chi_\e\) and by using that \(|\nabla \xi_\e|\leq C\) we have that \(|\xi_\e|\leq C a_\e/d_\e\) in \(\supp(1-\chi_\e)\cap A\). Hence, for \(\e\) sufficiently small, we obtain that \(d_\e f_\e\leq \frac{1}{4}\) in \(\supp(1-\chi_\e)\cap A\). Since \(v=1\) in \(\A\setminus A\), recalling the inequality \(|\nabla v|^2 \geq 2 |\p_xv\wedge \p_yv|=|\mu_v|\), proved by using the Cauchy-Schwarz inequality, we find that 
\begin{equation}\label{eq:(4.36)}
B \geq \int_{\A}(1-\chi_\e)\left[ \frac{\eta_\e^2}{4}|\nabla v|^2+\frac{\eta_\e^4}{4\e^2}(|v|^2-1)^2\right]\geq 0.
\end{equation}
Now we set \(\tilde{\e}:=\frac{\e}{\inf_{\A}\eta_\e}\). Since from Lemma \ref{lem:estimates_eta} we know that \(\eta_\e=1+O(d_\e^2 \e^2)\) we have that \(\tilde{\e}=\e+O(d_\e^2\e^3)\). We suppose that \(u=w_\e v\) minimizes \(E_\e\) in \(\I_\e\). Then we find that \(G_\e(v)\leq 0\) and this entails
\begin{equation}\label{eq:inegA1A2B}
A_1+B\leq A_2.
\end{equation} Since \(\frac{1}{\tilde{\e}^2}\leq \frac{\eta_\e^2}{\e^2}\) in \(\A\) we observe that 
\begin{equation}\label{**}
\int_{\A} \frac12 |\nabla v|^2+\frac{1}{4\tilde{\e}^2}(|v|^2-1)^2\leq (\inf_{\A}\eta_\e)^{-2}(A_1+2B)\leq CA_2.
\end{equation}

The key of the proof of Theorem \ref{th:uniqueness_d_small} is a weighted Jacobian estimate due to Jerrard that we recall now:

\begin{lemma}(\cite[Lemma 8]{Jerrard_2007} and \cite[Lemma 4.1]{Aftalion_Jerrard_RoyoLetelier_2011})\label{lem:weighted_jac_est}

There exists a universal constant \(C>0\) such that for any \(\kappa \in (1,2)\), open set \(U\subset \R^2\) and \(u\in H^1(U,\R^2)\), and \(\e'\in (0,1)\),
\begin{multline*}
\left|\int_U \phi \p_xu \wedge \p_y u\right|\leq \kappa \int_U |\phi| \frac{e_{\e'}(u)}{|\ln \e'|}\\
+C{\e'}^{(\kappa-1)/50}(1+\|\phi\|_{W^{1,\infty}})\left(\|\phi\|_{L^\infty}+1+\int_{\supp \phi}(|\phi|+1)e_{\e'}(u) \right)
\end{multline*}
for all \(\phi \in \C^{0,1}_c(U)\). Here \(e_{\e'}(u)=\frac12 |\nabla u|^2+\frac{1}{4{\e'}^2}(1-|u|^2)^2\).
\end{lemma}

 We use this lemma with \(U=\tilde{A}\) with, \(\kappa>1\) to be chosen later, with \(\e'=\tilde{\e}\) and with \(\kappa>1\). We obtain

\begin{equation*}
|A_2|=d_\e\left|\int_{\A} \chi_\e \xi_\e \mu_v  \right|\leq 2d_\e \kappa \int_{\A}\chi_\e \xi_\e \frac{e_{\tilde{\e}}(v)}{|\ln \tilde{\e}|}+\mathcal{E}_\e,
\end{equation*}
where
\begin{equation*}
\mathcal{E}_\e=C\tilde{\e}^{(\kappa-1)/50}(1+\|\chi_\e \xi_\e\|_{W^{1,\infty}})\bigl[\|\chi_\e\xi_\e\|_{L^\infty}+1+\int_{\supp \phi}(|\phi|+1)e_{\tilde{\e}}(v)\bigr].
\end{equation*}
We recall that \(\|\xi_\e\|_{L^\infty}+\|\nabla \xi_\e\|_{L^\infty}\leq C\) and, by construction, \(\|\chi_\e\|_{L^\infty}\leq C\) and \(\|\nabla \chi_\e\|_{L^\infty}\leq C\frac{d_\e}{a_\e}\). Since \(d_\e \leq C |\ln \e|\) we see  thanks to \eqref{**}, that  for \(\e\) sufficiently small
\begin{equation*}
\mathcal{E}_\e \leq C \e^\beta (1+|A_2|), \text{ for } \beta=(\kappa-1)/100.
\end{equation*}
Now, the choice of \(\tilde{\e}\) guarantees that 
\begin{equation*}
e_{\tilde{\e}}(v)\leq \frac12 |\nabla v|^2+\frac{\eta_\e^2}{4\e^2}(|v|^2-1)^2 \text{ in } \A.
\end{equation*}
We recall that \(\xi_\e=f_\e\eta_\e^2\) and we obtain
\begin{align*}
(1-C\e^\beta)|A_2| & \leq 2 d_\e\kappa \frac{\|f_\e\|_{L^\infty}}{|\ln \tilde{\e}|}\int_{\A} \chi_\e \left(\frac{\eta_\e^2}{2}|\nabla v|^2+\frac{\eta_\e^4}{4\e^2}(|v|^2-1)^2 \right)+C\e^\beta \\
& \leq 2d_\e \kappa \frac{\|f_\e\|_{L^\infty}}{|\ln \tilde{\e}|}A_1 +C\e^\beta.
\end{align*}
By using that \(\eta_\e=1+O((d_\e\e)^2)\), we can check that \(\|f_\e\|_{L^\infty}=\|\frac{\xi_\e}{\eta_\e}\|_{L^\infty}\leq \ln \frac{R_2}{R_1}+o_\e(1)\). We thus deduce that 
\begin{align*}
|A_2|\leq 2d_\e \frac{\ln \frac{R_2}{R_1}+o_\e(1)}{|\ln \e|+\ln (1+O_\e(\e^2))}\kappa A_1+C\e^\beta.
\end{align*}
Since we assume \(d_\e <\alpha_c |\ln \e|\) we can write that \(d_\e \leq \lambda \frac{|\ln \e|}{\ln (R_2/R_1)}\) with \(\lambda<1/2\). We choose \(\kappa\in (1,2)\) so that \(\lambda \kappa<1/2\). We find that, for \(\e\) small enough,
\begin{equation*}
|A_2|\leq \lambda'A_1+C\e^{\beta}, \text{ with } \lambda'<1.
\end{equation*}
By using that \(A_1\leq A_2\), from \eqref{eq:(4.36)} and \eqref{eq:inegA1A2B}, we obtain that \(A_1\leq C\e^{\beta}\). This shows that \(|A_2|\leq C\e^{\beta}\) and finally \(B\leq C\e^\beta\). In view of \eqref{**} we infer
\begin{equation*}
\frac12 \int_{\A}|\nabla v|^2+\frac{1}{4\tilde{\e}^2}\int_{\A}(1-|v|^2)^2\leq C\e^\beta.
\end{equation*}
Now, we recall that \(\tilde{\e}=\e(1+O((d_\e \e)^2)\). Furthermore, if \(u\) minimizes \(E_\e\) then \(|u|=\eta_\e |v|\leq 1\) and then \(|v|\leq 1/\eta_\e=1+O((d_\e \e)^2)\leq C\). Thus
\begin{align*}
\frac{1}{4\tilde{\e}^2}\int_{\A}(1-|v|^2)^2 &= \frac{1}{4\e^2}(1+O(d_\e^2))\int_{\A} (1-|v|^2)^2 \\
&=\frac{1}{4\e^2}\int_{\A} (1-|v|^2)^2+O(d_\e^2).
\end{align*} We conclude that 

\begin{equation*}\label{eq:bound_integral}
\frac12 \int_{\A}|\nabla v|^2+\frac{1}{\e^2}\int_{\A}(1-|v|^2)^2\leq C\e^\beta.
\end{equation*}

But, assuming that \(u\) minimizes \(E_\e\) in \(\I_\e\) we know that \(u\) satisfies \eqref{eq:Euler_Lagrange} and then we can show that \(\|\nabla u\|_{L^\infty(K)}\leq C/\e\) for all \(K\) compact sets included in \(\overline{B}_{R_2}\setminus \overline{B}_{R_1}\). This estimate, in the interior of the domain is classical cf.\ \cite[Step A.1]{BBH_93}. Near \(\p B_{R_2}\) since we have a Dirichlet boundary data we can follow the proof of \cite[Equation (33)]{BBH_93}.
 Since \(u=w_\e v\) and since \(|\nabla w_\e|\leq Cd_\e\) we obtain that \(\|\nabla v\|_{L^\infty}\leq \frac{C}{\e}\). This gradient estimate and \eqref{eq:bound_integral} imply that 
\begin{equation}\label{eq:unif_conv_v}
|v|\geq 1-C\e^{\frac{\beta}{2}} \text{ in every compact set } K \subset \overline{B}_{R_2}\setminus \overline{B}_{R_1}.
\end{equation}
A proof of this fact can be obtained by arguing by contradiction, cf.\ e.g.\ \cite[Step A.2]{BBH_93}. 
\medskip
Thus we have obtained that 
\begin{equation*}
 \frac{u_\e}{\eta_\e e^{id_\e\theta}} \rightarrow 1, \text{ in } \C^0_{\text{loc}}(\overline{B}_{R_2}\setminus \overline{B}_{R_1}).
 \end{equation*}
Recalling \eqref{eq:estimate_eta} we arrive at the conclusion of Theorme \ref{th:uniqueness_d_small}.

\medskip

\textbf{Remark:}  If we can show that \(\|\nabla u_\e\|_{L^\infty(\tilde{A})}\leq C/\e\) then the conclusion of Theorem \ref{th:uniqueness_d_small} can be strengthen in the following way: there exists \(\e_0>0\) such that if \(0<\e<\e_0\) and if \(d_\e< \alpha_c|\ln \e|\), then, \(u_\e =\eta_\e(r)e^{id_\e\theta}\) where \(\eta_\e\) is the minimizer of the energy \eqref{eq:energy_eta} among real-valued functions in \(H^1( (R_1,R_2))\) satisfying \(\eta(R_1)=\eta(R_2)=1\). This would be a uniqueness and symmetry result. the difficulty is to obtain the \(L^\infty\) bound on the gradient near the boundary \(\p B_{R_1}\) where \(u_\e\) satisfies only the ``semi-stiff'' boundary condition \(|u_\e|=1\) and \(u_\e \wedge \p_\nu u_\e=0\).

\section{The case \(d_\e \gg |\ln \e|\): proof of Theorem \ref{th:main_0_case_d_big}}\label{sec:d_big}
In this section we are interested in the case where the degree on the outer boundary of the annulus is much bigger than \(|\ln \e|\).

\begin{proof}(Proof of Theorem \ref{th:main_0_case_d_big})

By Item 3) of Proposition \ref{cor:upper_bound_final} we can say that
\begin{align*}
\pi d_\e |\ln \e|(1+o(1)) &\geq E_\e(u_\e)\geq \frac12 \int_A |\nabla u_\e|^2  \geq \frac12 \int_A |u_\e\wedge \nabla u_\e|^2 \geq \frac12 \int_A |j_\e|^2.
\end{align*}
This shows that \(j_\e/d_\e \rightarrow 0\) in \(L^2(A)\). Now for the extended function \(\tilde{u}_\e\) and for the extended quantities \(\tilde{j}_\e\) and \(\tilde{\mu}_\e\), this means that
\[\frac{\tilde{j}_\e}{d_\e} \rightarrow \tilde{j}:= \begin{cases} 0 & \text{ in } A \\
\frac{\vec{e_\theta}}{r} & \text{ in } \p B_{R_2}.\end{cases}\]
Thus, proceeding as in \eqref{eq:calcul_curl} we obtain that
 \[\curl \frac{\tilde{j}_\e}{d_\e}= \frac{(2\p_x u_\e\wedge \p_y u_\e)}{d_\e} \textbf{1}_{A}\rightharpoonup \curl \tilde{j}=\mathcal{H}^1_{\lfloor \p B_{R_2}} \text{ in } \mathcal{D}'(A_a).\]


\end{proof}

\section*{Appendix A: Green function in an annulus}\label{sec:AppendixA}

We prove in this Appendix that we can give an explicit expression of the Green function \eqref{eq:Green} in an annulus by means of series. This was done before for the Green function in an annulus with homogeneous Dirichlet boundary conditions in \cite{Hickey_1929}.

\begin{proposition}
Let \(G\) be the Green function in \(A_a\) defined in \eqref{eq:Green}. Then for all \(x,y\in A_a\) we have \(G(x,y)=-\frac{1}{2\pi}\ln |x-y|+\varphi (x,y)\) where, in polar coordinates \(\varphi(x,y)=\tilde{\varphi}(r,\theta, \rho,\psi)\) and, for \( R_1<r,\rho<R_2\) and \(\theta,\psi\in [0,2\pi]\) we have
\begin{multline*}
\tilde{\varphi} (r,\theta,\rho,\psi) = A_0+B_0 \ln r \\
+\sum_{m=1}^{+\infty} \bigl[ (r^mA_m+r^{-m}B_m)\cos (m\theta) +(r^mC_m+r^{-m}D_m)\sin (m\theta)\bigr],
\end{multline*}
with
\begin{equation*}
\begin{array}{rccll}
A_0&=&-\frac{1}{2\pi}\log \frac{\rho}{R_1}, \quad &B_0 =&\frac{-1}{2\pi}, \\
A_m&=& \frac{1}{2\pi m}\frac{\left(\frac{R_1}{\rho} \right)^m-\left(\frac{\rho}{R_1\tilde{R}_2^2}\right)^m }{R_1^m+R_1^{-m}}\cos (m\psi), \quad
&B_m =& \frac{1}{2\pi m}\frac{\left(\frac{R_1}{\rho} \right)^m+\left(\frac{R_1\rho}{\tilde{R}_2^2}\right)^m }{R_1^{m}+\tilde{R}_2^{2m}R_1^{-m}} \tilde{R}_2^{2m} \cos (m\psi), \\
C_m&=& \frac{1}{2\pi m}\frac{\left(\frac{R_1}{\rho} \right)^m-\left(\frac{\rho}{R_1\tilde{R}_2^2}\right)^m }{R_1^m+R_1^{-m}}\sin (m\psi),\quad
&D_m=& \frac{1}{2\pi m}\frac{\left(\frac{R_1}{\rho} \right)^m+\left(\frac{R_1\rho}{\tilde{R}_2^2}\right)^m }{R_1^{m}+\tilde{R}_2^{2m}R_1^{-m}} \tilde{R}_2^{2m} \sin (m\psi),
\end{array}
\end{equation*}
where we defined \(\tilde{R}_2:=R_2+a\). In particular \(\varphi \in \C^\infty(\overline{A}_a\times A_a)\) and \(\varphi \in \C^\infty( A_a\times \overline{A}_a)\).
\end{proposition}

\begin{proof}
Since \(\tilde{\varphi}\) is harmonic in \(A_a\) we can write it as 
\begin{multline*}
\tilde{\varphi}(r,\theta,\rho,\psi) =A_0+B_0\ln r +\sum_{m=1}^{+\infty} \Bigl[ (r^mA_m+r^{-m}B_m)\cos (m\theta) \\
 +(r^mC_m+r^{-m}D_m)\sin (m\theta) \Bigr].
\end{multline*}
The boundary condition satisfied by \(G\) on \(\p B_{R_1}\) implies that \(\varphi(x,y)=-\frac{1}{2\pi}\ln |x-y|\) for every \(x\in \p B_{R_1}\) and every \(y \in A_a\). By using that \(\ln |z| =\text{Re} (\ln z)\) for any \(z\in \mathbb{C}\setminus \R_{-}\) we can expand
\begin{align*}
-\ln |R_1e^{i\theta}-\rho e^{i\psi}|&=-\ln \rho -\RE \ln \left(1-\frac{R_1}{\rho}e^{i(\theta-\varphi)}\right) \\
&=  -\ln \rho +\RE \sum_{m=1}^{+\infty} \frac{1}{m}\left(\frac{R_1}{\rho} \right)^m e^{im(\theta -\psi)} \\
&=-\ln \rho+\sum_{m=1}^{+\infty}  \frac{1}{m}\left(\frac{R_1}{\rho} \right)^m \cos (m(\theta-\psi)).
\end{align*}
We thus find a first relation between the coefficients \(A_m,B_m,C_m,D_m\) which can be written as
\begin{align*}
A_0+B_0\ln R_1 &=-\frac{1}{2\pi}\ln \rho \\
R_1^mA_m+R_1^{-m}B_m &= \frac{1}{2\pi m}\left(\frac{R_1}{\rho} \right)^m \cos (m\psi) \\
R_1^mC_m+R_1^{-m}D_m &= \frac{1}{2\pi m}\left(\frac{R_1}{\rho} \right)^m \sin (m\psi).
\end{align*}
Now on \(\p B_{R_2+a}\) we must have \(\p_r \tilde{\varphi}(r,\theta, \rho \psi)=-\frac{1}{2\pi}\p_r \ln |re^{i\theta}-\rho e^{i\varphi}|\). To simplify notation we set \(\tilde{R}_2=R_2+a\). We can check that 
\begin{align*}
-\frac{1}{2\pi}\p_r \ln |re^{i\theta}-\rho e^{i\varphi}| =-\frac{1}{2\pi}\left(\frac{1}{r}+\sum_{m=1}^{+\infty} \frac{\rho^m}{r^{m+1}}\cos (m(\theta-\psi)) \right).
\end{align*}
This leads to a second relation between the coefficients.
\begin{align*}
B_0&= -\frac{1}{2\pi} \\
\tilde{R}_2^{m-1} A_m-mB_m\tilde{R}_2^{-m-1} &=\frac{\rho^m}{2\pi \tilde{R}_2^{m+1}}\cos(m\psi) \\
\tilde{R}_2^{m-1} C_m-mD_m\tilde{R}_2^{-m-1} &=-\frac{\rho^m}{2\pi \tilde{R}_2^{m+1}}\sin (m\psi). 
\end{align*}
To find the coefficients we thus have to solve the linear systems
\begin{equation*}
\left\{
\begin{array}{rcll}
R_1^mA_m+R_1^{-m}B_m &=\frac{1}{2\pi m} \left(\frac{R_1}{\rho} \right)^m \cos (m\psi) \\
m\tilde{R}_2^{m-1}-m\tilde{R}_2^{-m-1} &= -\frac{1}{2\pi} \frac{\rho^m}{\tilde{R}_2^{m+1}}\cos (m\psi) 
\end{array}
\right.
\end{equation*}
\begin{equation*}
\left\{
\begin{array}{rcll}
R_1^mC_m+R_1^{-m}D_m &=\frac{1}{2\pi m} \left(\frac{R_1}{\rho} \right)^m \sin(m\psi) \\
m\tilde{R}_2^{m-1}-m\tilde{R}_2^{-m-1} &= -\frac{1}{2\pi} \frac{\rho^m}{\tilde{R}_2^{m+1}}\sin (m\psi).
\end{array}
\right.
\end{equation*}
This leads to the desired formulas for \(A_m,B_m,C_m,D_m\). We point out that other derivations of Green functions in an annulus can be found with the help of Jacobi \(\theta\) functions, see e.g.\ in \cite{fetter67,Guenther_Massignan_Fetter_2017}.

The fact that \(\varphi \in \C^\infty(\overline{A}_a\times A_a)\) and \(\varphi \in \C^\infty( A_a\times \overline{A}_a)\) follows from the theorems of differentiation of series.
\end{proof}
\section*{Appendix B: Convergence of quadratic expressions}

In this Appendix we study the convergence of quadratic expressions of measures under weak convergence, see also \cite[Lemma 5.3]{Schochet_1995} and \cite[Lemma 6.3]{Peszek_Rodiac_2025},
\begin{proposition}\label{prop:Appendixb}
Let \(\Omega \subset \R^N\) be an open set. Let \(K:\Omega \times \Omega \rightarrow \R\) be such that \(K\in \C((\Omega \times \Omega) \setminus \D)\), where \(\D:=\{(x,y)\in \R^N\times \R^N; x=y\}\), and such that \(K\in L^\infty (\Omega \times \Omega)\). Let \( (\mu_n)_n\subset \mathcal{M}(\Omega)\) be a sequence of Radon measures such that \(\mu_n \rightharpoonup \mu\) in the sense of measures. Assume furthermore that \(\mu\) is non-atomic and that \(|\mu_n|\rightharpoonup |\mu|\) in the sense of measures. Then
\begin{equation}\label{eq:conv_quadratic}
K(\cdot,\cdot)\mu_n \otimes \mu_n \rightharpoonup K(\cdot,\cdot)\mu \otimes \mu \quad \text{ in }  \M(\Omega \times \Omega).
\end{equation}
\end{proposition}

\begin{proof}
By density  in \(\C_c(\Omega\times \Omega)\) of functions of the form \((x,y)\mapsto \sum_{k,l} f_k(x)g_l(y)\) for \(f_k,g_l\in \C_c(\Omega)\), we obtain that \(|\mu_n| \otimes |\mu_n| \rightharpoonup |\mu| \otimes |\mu|\) in \(\M(\Omega\times \Omega)\).  We also have that 
\( \left| K(\cdot,\cdot)\mu_n\otimes \mu_n\right| \leq \|K\|_{L^\infty}|\mu_n| \otimes |\mu_n|\). Hence we see that \(K(\cdot,\cdot)\mu_n\otimes \mu_n\) is uniformly bounded and, up to a subsequence, it converges to a measure \(m\) in the sense of measures. We can also apply Lemma \ref{lem:th_measure} below to obtain that \(m\) is absolutely continuous with respect to \(|\mu|\otimes |\mu|\).
Now we observe that, since \(K \in \C((\Omega \times \Omega) \setminus \D)\), for every \(\varphi\in \C_c(\Omega \times \Omega)\) vanishing on \(\D\) we have that
\begin{align*}
\iint_{\O\times \O}\varphi(x,y) K(x,y)d\mu_n(x)d\mu_n(y) \xrightarrow[n \to +\infty]{} \iint_{\O\times \O}\varphi(x,y) K(x,y)d\mu(x)d\mu(y).
\end{align*}
Let \(\eta \in \C^\infty_c(\R^N)\) be such that \(\eta \equiv 1\) in a neighbourhood of \(\D\). Then for all \(\varphi \in \C_c(\Omega \times \Omega)\), and for all \(\kappa>0\)
\begin{align*}
&\iint_{\O\times \O}\varphi(x,y) K(x,y)d\mu_n(x)d\mu_n(y) = \iint_{\O\times \O}\varphi(x,y) K(x,y)\eta\left(\frac{|x-y|}{\kappa}\right)d\mu_n(x)d\mu_n(y) \\
&  \quad + \iint_{\O\times \O}\varphi(x,y) K(x,y)\left[1-\eta\left(\frac{|x-y|}{\kappa}\right)\right] d\mu_n(x)d\mu_n(y) \\
& \xrightarrow[n\to +\infty]{}\iint_{\O\times \O}\varphi(x,y) K(x,y)\eta\left(\frac{|x-y|}{\kappa}\right)dm(x,y) \\
 &\quad + \iint_{\O\times \O}\varphi(x,y) K(x,y)\left[1-\eta\left(\frac{|x-y|}{\kappa}\right)\right] d\mu(x)d\mu(y)  \\
& \xrightarrow[\kappa \to 0]{}\iint_{\D}\varphi(x,y) K(x,y)dm(x,y) \\
& \quad + \iint_{\O\times \O}\varphi(x,y) K(x,y) d\mu(x)d\mu(y).
\end{align*}
In the last step we have used the dominated convergence theorem. Now we observe that since \(|\mu|\) is non-atomic then \(|\mu|\otimes |\mu|\) does not charge the diagonal \(\D\), see e.g.\ \cite{Delort_1991} p.566. Indeed , since \(|\mu|(y)=0\) for all \(y \in \Omega\), Fubini's theorem shows that 
\begin{align*}
\iint_{\Omega \times \Omega} \textbf{1}_{x=y} d|\mu|(x)d|\mu|(y) &= \int_{\Omega} d|\mu|(y)\int_{\Omega}\textbf{1}_{\{x=y\}} d |\mu|(x) \\
&=\int_{\Omega}|\mu|(y) d|\mu|(y)=0.
\end{align*}

But since \(m\) is absolutely continuous with respect to \(|\mu|\otimes |\mu|\), we also have that \(m\) does not charge \(\D\) and then
\(\iint_{\D}\varphi(x,y) K(x,y)dm(x,y)=0\). This proves that \eqref{eq:conv_quadratic} holds.
\end{proof}
In the above proof we have used the following result, whose proof can be found in \cite[Lemma 5.1]{Schochet_1995}

\begin{lemma}\label{lem:th_measure}
Let \(\Omega\) be an open set of \(\R^N\). Let \( (\nu_n)_n\subset \M(\Omega)\) and \((\sigma_n)_n\subset \M(\O)^+\) be such that there exists \(C>0\) with \(|\nu_n|\leq C \sigma_n\) for all \(n\in \mathbb{N}\). Assume that \(\nu_n \rightharpoonup \nu\) in the sense of measures and \(\sigma_n \rightharpoonup \sigma\) in the sense of measures. Then there exists a \(\sigma\)-measurable function \(f:\Omega\rightarrow \R\) such that \(\nu=f \sigma\) and \(|f|\leq C\) for \(\sigma\)-a.e. point in \(\O\).
\end{lemma}

\bibliographystyle{abbrv}
\bibliography{bib}
\end{document}